\documentclass[smallextended,referee,envcountsect]{svjour3}
\smartqed
\usepackage{graphicx}
\journalname{JOTA}

\usepackage[square,sort,comma,numbers]{natbib}
\usepackage{etoolbox}
\makeatletter
\patchcmd{\appendix}{\small}{}{}{}
\makeatother

\usepackage{import}
\usepackage{latexsym,amssymb,amsmath,amsfonts}

\usepackage{amsthm}

\usepackage[dvipsnames]{xcolor}
\usepackage{manyfoot}
\usepackage{graphicx,shadow}
\usepackage{float}

\usepackage{mathrsfs}
\usepackage{caption,subcaption}
\usepackage{soul}
\usepackage{colortbl} 
\usepackage{multirow} 
\usepackage{epstopdf}
\usepackage[colorlinks,linkcolor=blue,citecolor=blue,urlcolor=blue]{hyperref}
\usepackage[noabbrev,capitalise,nameinlink]{cleveref}
\usepackage[algo2e,boxed,ruled,linesnumbered,nosemicolon]{algorithm2e}
\crefname{algocf}{alg.}{algs.}
\Crefname{algocf}{Algorithm}{Algorithms}
\usepackage{enumitem}
\usepackage{tikz}
\usetikzlibrary{positioning}
\usepackage[noabbrev,capitalise,nameinlink]{cleveref}

\usepackage{comment}

\def\prox{\displaystyle\mathop {\mbox{\rm prox}}}      

\def\dom{\displaystyle\mathop {\mbox{\rm dom}}}    

\def\sup{\displaystyle\mathop {\mbox{\rm sup}}}

\def\inf{\displaystyle\mathop {\mbox{\rm inf}}}

\def\argmin{\displaystyle\mathop {\mbox{\rm argmin}}}

\def\real{\mathbb R}

\newtheorem{thm}{Theorem}
\newtheorem{lem}{Lemma}
\theoremstyle{definition}
\newtheorem{assumption}{Assumption}
\newtheorem{defin}{Definition}

\newcommand{\bx}{{\bf x}}

\newcommand{\bz}{{\bf z}}

\DeclareMathOperator{\crit}{Crit}

\begin{document}



\title{An Augmented Lagrangian Approach to Composite Problems with a Random Linear Operator}
\titlerunning{AL Approach to Composite Problems with a Random Linear Operator}


\author{Dan Greenstein \and Nadav Hallak}

\institute{Dan Greenstein \at
            Technion Israel Institute of Technology \\
            Technion City, Haifa 3200003, Israel\\
              sdngreen@campus.technion.ac.il
              \and
              Nadav Hallak \at
              Technion Israel Institute of Technology \\
              Technion City, Haifa 3200003, Israel\\
              ndvhllk@technion.ac.il
}

\date{Received: date / Accepted: date}

\maketitle

\abstract{%
    This paper studies nonconvex composite problems comprising the sum of a smooth function and the composition of an extended real-valued nonconvex function applied to a stochastic linear operation  accessible only via a sampling process.
    The generality of the model allows it to capture  applications in various fields, especially those involving stochastic nonconvex compositions such as nonconvex expectation constraints and nonconvex nonsmooth regularizers applied to stochastic linear mappings.  
    We develop an Augmented Lagrangian-based scheme that utilizes the interplay between three modular elements: a sampling regime, an alternating Augmented Lagrangian update, and an Augmented Lagrangian penalty adaptation procedure, to achieve convergence guarantees of the scheme for the underlying problem.
    The universality of our approach allows flexibility in its implementation, and  we further fine-tune its general form to exploit structural assumptions.
}%
\par\noindent Communicated by Xiaolu Tan.



%

\section{Introduction}
\subsection{Problem Formulation}
This paper addresses the nonconvex nonsmooth composite model with a random linear operator  given by
\begin{equation}\label{eq:9}
\tag{P}
  \min\limits_{x\in\real^n} h(x)+P(\mathbb{E}[M]x),
\end{equation}

where $h: \real^n \rightarrow \real$ is a twice continuously differentiable function, $P: \real^m \rightarrow \real \cup \{\infty\}$ is a proper, lower semi-continuous (l.s.c.)  prox-tractable function, $\mathbb{E}[\cdot]$ is an entrywise expectation operator, and  $M:\real^n\rightarrow\real^m$ is a random linear map with distribution $\pi$ having a surjective expectation matrix that is only accessible  via a sampling procedure.

The structure of \eqref{eq:9} is exceedingly general due to the composition of the \textit{nonconvex extended real-valued} function $P$, as it can model smooth/nonsmooth nonconvex regularizers and constraints, with the random mapping $M$ that can capture realistic scenarios in which the data is sampled in a streaming manner during the optimization process (see, e.g., \cite{hoi2021online}).
This positions  \eqref{eq:9} at the intersection of three branches within the optimization literature:  Lagrangian-based methods, stochastic nonconvex optimization, and optimization with nonconvex expectation constraints. 
Notably, in the context of \eqref{eq:9}, the literature in each branch is lacking, with the first being restricted to \textit{deterministic} nonconvex composite operations (\cite{LP15}, \cite{BN20}, \cite{bolte2018nonconvex}, \cite{CHT2021}, \cite{hallak2023adaptive}),  the second to nonconvex \textit{real-valued} compositions (\cite{wang2017accelerating}, \cite{ghadimi2020single}, \cite{ruszczynski2021stochastic}, \cite{chen2021solving}, \cite{jiang2022optimal}, \cite{xiao2022projection}), and the third to constraints that are differentiable, weakly convex, or a combination of a differentiable nonconvex function with a nonsmooth convex function (\cite{liu2019stochastic}, \cite{ma2019proximally}, \cite{boob2023stochastic}, \cite{li2023stochastic}).
To the best of our knowledge, our work is the first to address the composition of a \textit{stochastic mapping} with an \textit{extended real-valued} function that captures this level of  generality of the  constraints.


Due to this general structure,  a broad range of problems and applications can be modeled by \eqref{eq:9}, such as online linear embedding techniques like Incremental Principal Component Analysis \cite{ross2008incremental}, Online PCA \cite{boutsidis2014online}, and Online ICA \cite{Akhtar2012onlineICA}, utilized in a setup in which data arrives as a stream of i.i.d. linear embeddings $(M_1, M_2, \ldots)$, and the goal is to only use a  subset of the embedded dimensions.   
That is, $\min\limits_{x\in \real^n} h(x) \text{ subject to } \|\mathbb{E}[M]x\|_0 \le k. $

Another example is Blind Image Deconvolution \cite{campisi2017blind}, such as that applied in remote sensing, medical ultrasound, and optical telescope systems, in which the blur affecting the image
is estimated via sampling -- see, e.g.,  \cite[Section 1.3]{campisi2017blind} and \cite{fergus2006removing}. It is often modeled by 
variations of the  Tikhonov filtering model \cite{tikhonov1943stability} captured by \eqref{eq:9}, 
 $\hat{x} = \argmin_x \{\|\mathbb{E}[H]x - y\|^2 + \lambda\|x\|^2\}. $

We also note Inverse Reinforcement Learning models \cite{arora2021survey}, which can also be captured by \eqref{eq:9}, such as the ones suggested in \cite{leibovich2022learning} when the  approximation of the policy is linear with respect to the policy parameters.

To tackle the inherent difficulties following from the  nonconvex extended real-valued composition applied to a random map, we combine three mechanisms: a sampling procedure, an alternating Augmented Lagrangian (AL) update, and an  AL penalty adaptation. 
These, together with additional technical elements,  form our proposed \textit{Iterative Sampling Alternating Directions (ISAD)} method  outlined below; its detailed schematics are given in \Cref{sec: Iterative Sampling Alternating Direction Method of Multipliers}.

%
%

\medskip

\noindent\fbox{%
    \parbox{\dimexpr\textwidth-2\fboxsep-2\fboxrule\relax}{%
        \noindent \textbf{ISAD (\Cref{alg:1}) outline.}\\
Repeat:
\begin{enumerate}
    \item \textbf{Sampling mechanism (\Cref{definition:2}, \Cref{sec: Properties of the Sampling Mechanism}):} Sample $M$ to maintain an unbiased estimator of $\mathbb{E}[M]$;
    \item \textbf{Alternating Augmented Lagrangian update (\Cref{alg:ALAS}, \Cref{sec: Iterative Sampling Alternating Direction Method of Multipliers}):}  An alternating directions update;
    \item \textbf{Penalty adaptation (\Cref{adaptive_penalty_oracle_requirements_defin},  \Cref{general_beta_oracle}):} Update the penalty parameter as needed.
\end{enumerate}
            }%
}

\medskip

These three ingredients  are  introduced in more detail in \Cref{sec:Mathematical Ingredients}.
Our theoretical analysis uses the interplay between these ingredients to establish the ISAD's convergence guarantees.

%
%

\subsection{Literature Review} 
Our model and  approach are  addressed by three branches in the literature:  Lagrangian methods,  stochastic  optimization, and  expectation constraints.

{\emergencystretch=5pt
\noindent \textbf{Augmented Lagrangian literature.}
The literature studying composite problems via an AL approach can be  divided according to the convexity, or the lack of convexity, of the composite function. 
To the best of our knowledge, our work is the first to address a composite problem with a random linear composition via the AL regardless of the convexity of the composite function.  
Even if we only consider expectation constraints, the only work we are aware of that uses AL is \cite{zhang2022solving} studying a fully convex model.
\par}


Problems with a convex composite structure are well-studied in the AL literature, but differ significantly in methodology and analysis from the nonconvex setup;  for a comprehensive review and references see \cite{SABACH2019401}.

In contrast, the AL literature on nonconvex composition is sparse and primarily focuses on attaining global convergence, i.e., the criticality of the accumulation points.
The work \cite{LP15} is the first to propose and analyze a method for addressing the nonconvex composite linear model, achieving global convergence under coerciveness and semialgebraic assumptions. 
Our approach, specifically the AL update in ISAD, draws inspiration from it. 
While \cite{LP15} relies on the surjectivity of the linear mapping and a large penalty parameter, we replace the latter with ISAD's penalty adaptation, avoiding the need to select an optimal penalty. 

Expanding on \cite{LP15}, the work in \cite{BN20} explores a model where the differentiable component is only once-differentiable. 
The paper proposes a meta algorithm  focusing on linear mappings with similar assumptions, and achieving global convergence under the Kurdyka--{\L}ojasiewicz (KL) assumption.

Turning to broader generalizations, \cite{bolte2018nonconvex} examines a broad class of problems using a general composite operator that captures most deterministic models, achieving global convergence under the KL assumption. 
However, this generalization requires an ``information zone,'' where the Jacobian of the nonlinear mapping operator is surjective, and parameters from this zone are incorporated into the penalty adaptation, which complicates the approach. 
We adopt their adaptive penalty concept but update it based on empirical Lipschitz continuity, avoiding the need for these specific parameters.

The work \cite{CHT2021} studies an extended composite model that allows for easy inclusion of constraints independent of the mapping by assuming Lipschitz continuity of the composition element. 
The authors propose a splitting algorithm with an adaptive penalty determined by a backtracking procedure to attain global convergence under the KL assumption.
  
More recently, the work \cite{hallak2023adaptive} was the first to overcome the surjective Jacobian requirement. 
The authors study a model similar to \cite{bolte2018nonconvex} and propose an adaptive approach to obtain  approximate critical points, and subsequently, critical points.
However, unlike other works assuming surjectivity, \cite{hallak2023adaptive} does not guarantee global convergence.

We emphasize that none of the aforementioned studies can address our underlying model due to the unknown random matrix element.

\noindent \textbf{Nonconvex Stochastic Composite Optimization literature.} 
In contrast to the AL literature, the papers exploring gradient and subgradient approaches to composite optimization extensively reference random mappings. However, it is important to highlight that these referenced works exhibit notable distinctions from our research along various dimensions.

The studies \cite{wang2017accelerating, ghadimi2020single} investigate nonconvex stochastic composition models, differing from our work primarily in their differentiability assumptions and their focus on real-valued functions. 
These distinctions constrain their models' ability to handle nonconvex  noncontinuous constraints, compared to ours.



The paper \cite{ruszczynski2021stochastic} investigates a nested composition model with multiple layers applied to nonconvex, non-differentiable functions. The assumptions include access to Clarke subdifferential estimates, a compact feasible set, Lipschitz continuity, and bounded subgradients. Key differences from our work lie in these assumptions and the use of composition with real-valued functions, which limits the encoding of nonconvex and noncontinuous constraints.

Moreover, works such as \cite{chen2021solving}, \cite{jiang2022optimal} and \cite{xiao2022projection} focus on a stochastic nonconvex nested composition model with rates of convergence but assume smoothness in the composing functions. This, again, diverges from our work in a manner akin to our distinctions from \cite{wang2017accelerating}.


\noindent \textbf{Nonconvex Expectation Constraints literature.}
The nonconvex expectation constraints literature primarily focuses on constraints that are either differentiable, weakly convex, or a combination of a differentiable nonconvex function with a nonsmooth convex function. 
In contrast, our model can encode any linear composition of a prox-tractable function. 
Additionally, unlike the models in the nonconvex expectation constraints literature, our model can encode lower semi-continuous functions.

Notable related works are \cite{liu2019stochastic,ma2019proximally,boob2023stochastic,li2023stochastic}, where they all study 
a model with a 
differentiable nonconvex, or weakly-convex, expectation function with varying assumptions on the constraints and regularizers.
The approach in \cite{liu2019stochastic,ma2019proximally,boob2023stochastic} is based on solving a sequence of relaxed, or approximating, convex problems, to achieve first-order conditions.
The work \cite{li2023stochastic} proposes an algorithm to find  approximated KKT points, improving the convergence rate over \cite{ma2019proximally} and \cite{boob2023stochastic}.

\subsection{Contribution}
We briefly summarize our main contributions below.
\begin{itemize}[label=--]
	\item \textbf{Innovative Meta-Algorithm Design:} We introduce a meta-algorithm to address the previously unexplored stochastic optimization problem \eqref{eq:9}, involving the composition of a nonconvex lower semicontinuous function with a random linear mapping. Our approach ensures almost sure subsequence convergence to a critical point while offering flexibility in both penalty adaptation and the choice of the Bregman divergence factor. 
	In particular, when $h$ is smooth, we provide a simplified scheme with efficient computations.

    

	\item \textbf{Empirical Lipschitz-Based Implementation:} We present an effective and practical implementation strategy to tackle the highly general \eqref{eq:9}, leveraging our empirical Lipschitz constant to enhance the algorithm's efficiency and adaptability.  

\end{itemize}

\subsection{Notation and Basic Definitions}
We use the common notation used in the continuous optimization literature, and in particular in \cite{rockafellar2009variational}.
For  $f: \real^n \rightarrow \real \cup \{\infty\}$, and  $\vartheta > 0$, the proximal operator is given by
\begin{equation}\label{eq:prox}
    \prox\limits_{\vartheta f}(x) = \argmin\limits_{z \in \real^n} f(z) + \dfrac{1}{2\vartheta} \left\Vert z - x\right\Vert^2.
\end{equation}
Throughout this work, it is assumed that all the proximal operations are well-defined, and when the result is a set, there is a well-defined rule on how to choose an element from it. 
We say that $f$ is \textit{prox-tractable}  if for all $x \in \real^n$ and all $\vartheta > 0$, \eqref{eq:prox} is nonempty  and there exists an efficient algorithm for finding one of its elements.
See \cite{B17} for notable examples of  {prox-tractable} functions.

We denote $z\xrightarrow{f} x$ if $z\rightarrow x$ and $f(z)\rightarrow f(x)$. 
In this work, we use the general subdifferential (see \cite[Definition 8.3]{rockafellar2009variational}). The general subdifferential set of $f$ at point $x \in \dom{f}$ is denoted by $\partial f(x)$, and is given by
\begin{equation}\label{Subdifferential_def}
\begin{aligned}
	\partial f(x) \equiv \Biggl\{v \in \real^n : {} & \exists x^t \xrightarrow{f} x, v^t \rightarrow v, \\
	& \liminf_{z\rightarrow x^t} \dfrac{f(z)-f(x^t) - \langle v^t, z - x^t \rangle}{\|z-x^t\|} \ge 0 \ \forall t\Biggr\}.
\end{aligned}
\end{equation}
As in \cite{rockafellar2009variational}, we refer to the general subdifferential simply as the subdifferential.

For a twice continuously differentiable function $\phi: \real^n \rightarrow \real$, the Bregman distance $D_{\phi}: \real^n \times \real^n \rightarrow \real$ is defined as: $D_\phi (x_1, x_2) := \phi(x_1) - \phi(x_2) - \langle \nabla \phi(x_2), x_1 - x_2 \rangle. $
We denote  $[n] =\{1,\ldots,n\}$.
Unless stated otherwise,  $\|x\|$ refers to the $\ell_2$ norm when applied to $x \in \real^n$, and $\|X\|$  refers to the $\ell_2$ induced norm $\|X\|_{2,2}$ when applied to $X \in \real^{m \times n}$. 

\section{Mathematical Ingredients}
\label{sec:Mathematical Ingredients}

\cref{tab:notation} summarizes the main notation used throughout the algorithm and its analysis.
\begin{table}[htbp]
\centering
\caption{Main Notation}
\label{tab:notation}
\begin{tabular}{cl}
\hline
\textbf{Symbol} & \textbf{Description} \\
\hline
\multicolumn{2}{l}{\textit{Random Matrix \& Sampling}} \\
$M$ & Random linear map with distribution $\pi$ \\
$M^i$ & $i$-th independent sample of $M$ \\
$\bar{M}^t$ & Unbiased estimator of $\mathbb{E}[M]$ at iteration $t$ \\
$\delta_t$ & Estimation error: $\delta_t = \bar{M}^t - \mathbb{E}[M]$ \\
$\theta_t$ & Total number of samples collected by iteration $t$ \\[0.5em]
\multicolumn{2}{l}{\textit{Algorithm Variables}} \\
$x_t, y_t, z_t$ & Primal, auxiliary, and dual variables at iteration $t$ \\
$\beta_t$ & Penalty parameter at iteration $t$ \\
$g^{t+1}(x)$ & Subproblem function for $x$-update \\[0.5em]
\multicolumn{2}{l}{\textit{Key Constants}} \\
$\sigma$ & Minimum eigenvalue: $\sigma = \lambda_{\min}(\mathbb{E}[M]^T \mathbb{E}[M])$ \\
$\tilde{\sigma}_t$ & Minimum eigenvalue: $\tilde{\sigma}_t = \lambda_{\min}((\bar{M}^t)^T \bar{M}^t)$ \\
$K_{stable}$ & Random iteration after which $\beta_t$ stabilizes \\[0.5em]
\multicolumn{2}{l}{\textit{Empirical Lipschitz Constants}} \\
$L^e_F(\{x_t\})$ & Empirical Lipschitz constant of $F$ w.r.t. $\{x_t\}_{t\geq 0}$ \\
$\mu$ & Empirical Lipschitz constant: $\mu = L^e_{\nabla h + \nabla \phi}(\{x_t\})$ \\
$\nu$ & Empirical Lipschitz constant: $\nu = L^e_{\nabla \phi}(\{x_t\})$ \\
\hline
\end{tabular}
\end{table}

\subsection{The sampling mechanism}
 \label{sec:Properties of the linear map}
 Throughout the paper we assume that \Cref{alg:1} accesses a stream of unbiased i.i.d. estimators of $\mathbb{E}[M]$, which has full rank.
 \begin{defin}\label{entrywise_expectation_and_variance}
 	Let $M \sim \pi$, $M \in \real^{m \times n}$ be a random linear map. Then,
 	\begin{enumerate}
 		\item $\mathbb{E}[M] \in \real^{m \times n}$ is the entrywise expectation, given by $\left[\mathbb{E}\left[M\right]\right]_{i,j} = \mathbb{E}[M_{i,j}]. $
		\item $\mathrm{Var}[M] \in \real^{m \times n}$ is the entrywise variance, given by
		\begin{equation*}
			\left[\mathrm{Var}\left[M\right]\right]_{i,j} = \mathrm{Var}[M_{i,j}].
		\end{equation*}
 	\end{enumerate}
 \end{defin}
 \begin{assumption}[blanket assumption]\label{assum:blanket}
 \begin{enumerate}
     \item The sequence $\{M^i\}_{i>0}$ sampled by \Cref{alg:1} is of independent and identically distributed matrices;
     \item For all $i$, $M^i$ is a map from $\real^n$ to $\real^n$, meaning that $m = n$;
     \item $\mathrm{Var}[M]$ is bounded; and
     \item $\mathbb{E}[M]$ has full rank, meaning that there is a scalar $\sigma > 0$, such that $\lambda_{\min}(\mathbb{E}[M]^T \mathbb{E}[M]) = \sigma$.
 \end{enumerate}
 \end{assumption}
 
\begin{remark}
	 The model requires $\mathbb{E}[M]$ to be surjective, and therefore $m \le n$. The equality $m = n$ in the blanket assumption is without loss of generality. If $m < n$, a surjective map from $\real^n$ to $\real^m$ extends to a bijective map from $\real^n$ to $\real^n$ with probability $1$, by appending randomly generated rows and updating the function $P$ accordingly; see the detailed description in \cite{greenstein2023augmented}. From this point on we work under $m = n$, for which surjectivity is the full-rank condition in that assumption.
\end{remark}

The sampling efficiency of our algorithm may significantly be improved under a \textit{sub-Gaussian} assumption; see \cite[Sections 2.5--2.7]{vershynin_2018} for detailed definitions and properties, especially Section 2.7.1 therein.

\begin{defin}[Sub-Gaussian random variable]\label{subgaussian_random_variable_defin}
    Let $X$ be a random variable. If there exists $K>0$ such that the moment generating function of $X$ satisfies $ \mathbb{E}[\exp{X^2 / K^2}] \le 2, $
    then $X$ is called a sub-Gaussian random variable, and its norm is given by  $\|X\|_{\psi_2} = \inf \{t > 0 \; : \; \mathbb{E}[\exp{X^2 / t^2}] \le 2\},$
    where the notation $\|\cdot\|_{\psi_2}$ refers to the fact that it is an Orlicz norm for the function $ \psi_2(u) = \exp \{u^2\} - 1$ (cf. \cite[Section 2.7.1]{vershynin_2018}).
\end{defin}



\begin{assumption}[optional sampling assumption]\label{assum:4} For any $i,j \in [n]$ and any $t \ge 0$, the random variable $M^t_{i,j}$ sampled by \Cref{alg:1} is sub-Gaussian.
\end{assumption}
\Cref{assum:4}  can equivalently be expressed in the more common form: for any  $u,v \in \real^n$ such that $\left\Vert v\right\Vert = \left\Vert u\right\Vert = 1$, $v^T \left(M^t\right) u $ is a sub-Gaussian random variable.


It should be emphasized that \Cref{assum:4} is very general. 
Any random variable whose tail distribution can be bounded by a Gaussian random variable (up to a constant) is sub-Gaussian. 
This includes all bounded random variables, such as Bernoulli and Rademacher distributed random variables, and the Gaussian distribution. 



The sampling mechanism of the ISAD is designed so that $(\theta_{t+1}-\theta_t)$ matrices $\{M^i\}_{i=\theta_t + 1}^{\theta_{t+1}}$ are sampled at iteration $t$, 
where the sampling regime determines the values of $\{ \theta_t\}_{t\geq 0}$; we arbitrarily choose its value, but it can be changed by a positive factor.
The exact details including theoretical implications of the sampling regime are deferred to  \Cref{sec: Properties of the Sampling Mechanism}.
\begin{defin}[Sampling regime]\label{definition:2}
Let $\epsilon > 0$ be an arbitrary positive constant. 
 The sampling regime is defined as follows: If \Cref{assum:4} holds true, set $\theta_t = t^{1+\epsilon}$. 
 Otherwise, set $\theta_t = t^{2+\epsilon}$.
\end{defin}

\subsection{The Alternating Augmented Lagrangian}
The second element of ISAD is the alternating AL update.
Due to the combination of nonconvexity and nonsmoothness in  \eqref{eq:9}, we seek to obtain points satisfying the so-called \textit{general Fermat condition} (see \cite[Theorem 10.1]{rockafellar2009variational}), usually termed in the literature as \textit{criticality}. 
This rule states that any local minimum is a critical point, where
the set of critical points of $\psi: \real^n \rightarrow \real\cup\{-\infty,\infty\}$ is defined via a general subdifferential set \eqref{Subdifferential_def}  by $\crit{\psi} = \{x \in \real^n \; | \; 0 \in \partial \psi(x)\}.  $ 

Furthermore, due to the nonconvex nonsmooth composite structure of \eqref{eq:9}, we propose an alternating Augmented Lagrangian (AL) framework (see, e.g., \cite[Section 3.2.1]{B99}). 
To that end, let us first restate \eqref{eq:9} as follows: 
\begin{align}\label{eq:26}
	\min\limits_{x\in\real^n} \{ h(x)+P(y) : \text{ subject to }\mathbb{E}[M]x = y \}. 
\end{align}
The Augmented Lagrangian of \eqref{eq:26} with a penalty parameter $\beta>0$ is then given by
\begin{equation}\label{eq:25}
	\mathcal{L}_\beta(x,y,z) = h(x) + P(y) - \langle z, \mathbb{E}[M] x\rangle + \langle z, y\rangle + \dfrac{\beta}{2}\|\mathbb{E}[M]x - y\|^2.
\end{equation}
Note that the optimizer \textit{does not} have access to $\mathbb{E}[M]$, and therefore, in the optimization process it is replaced by an \textit{online estimator} that is updated according to the sampling regime.

The alternation scheme we propose updates the decision variables $(x,y,z)$ \textit{based on an online estimator of $\mathbb{E}[M]$} using the AL operator $\mathcal{A}(\cdot)$ whose pseudocode is described by \Cref{alg:ALAS}. 

The $y$ update can be efficiently calculated since we assume that $P$ is {prox-tractable}.
To calculate the $x$ update, we follow \cite{LP15} and utilize a Bregman divergence factor to allow flexibility in the update procedure in  \Cref{alg:ALAS}.
In particular, when $h$ has $\gamma$-Lipschitz gradient, $\phi(x) = \dfrac{\gamma}{2}\|x\|^2 - h(x)$ can be chosen, which results in $D_\phi(x,x_t) = -h(x) + \dfrac{\gamma}{2}\|x-x_t\|^2 + \langle \nabla h(x_t), x - x_t \rangle + h(x_t). $
Therefore, the $x$ update is reduced to
\begin{equation*}
    \argmin\limits_x \{- \langle z_t, \bar{M}^{t+1} x \rangle + \dfrac{\beta_t}{2}\| \bar{M}^{t+1} x - y_{t+1}\|^2 + \dfrac{\gamma}{2} \|x-x_t\|^2 + \langle \nabla h(x_t), x - x_t \rangle\},
\end{equation*}
where the transition from $\crit_x$ to $\argmin\limits_x$ is due to the strong convexity of the resulting subproblem.


\noindent	\begin{algorithm2e}[htbp]
	\caption{AL Alternation Scheme $\mathcal{A}(\cdot)$ }\label{alg:ALAS}
	\KwIn{$x_t, z_t \in\real^n$, $\bar{M}^{t+1}\in \real^{n\times n}$,  $\beta_t>0$, $\phi: \real^n \rightarrow \real$ where $\phi\in C^2$ and convex.} 
	set $y_{t+1} \in \argmin\limits_y \{P(y) + \langle {z_t}, y \rangle +   \dfrac{\beta_t}{2} \left\Vert{\bar{M}^{t+1} x_t- y} \right\Vert^2 \}$\;
		
	set	$x_{t+1} \in \crit\limits_x \{h(x) - \langle {z_t}, \bar{M}^{t+1} x \rangle + \dfrac{\beta_t}{2} \left\Vert{\bar{M}^{t+1} x- y_{t+1}} \right\Vert^2 + D_{\phi}(x,x_t) \}$\;

	set	$z_{t+1}  = z_t - \beta_t \left(\bar{M}^{t+1} x_{t+1} -y_{t+1}\right)$\;
	\textbf{return}	{$(y_{t+1},x_{t+1},z_{t+1})$}\;
		
%
%
%
%
\end{algorithm2e}

The AL update formulas imply the key relations stated in \Cref{lem:4}. 
\begin{lem}\label{lem:4}
	Let $\{x_t, y_t, z_t\}_{t \ge 0}$ be the sequence generated by \Cref{alg:1}. 
	Then for any $t\geq 0$ it holds that:
	\begin{enumerate}
		\item $0 \in \partial P(y_{t+1}) + z_{t+1} + \beta_t \bar{M}^{t+1} (x_{t+1} - x_t)$.
		
		\item $\nabla h(x_{t+1}) - (\bar{M}^{t+1})^T z_{t+1} = - \nabla \phi (x_{t+1}) + \nabla \phi (x_t)$.
		
		\item $\bar{M}^{t+1} x_{t+1} - y_{t+1} = \dfrac{1}{\beta_t} (z_t - z_{t+1})$.
	\end{enumerate}
\end{lem}
\begin{proof} 
	The proof  readily follows from the  first-order conditions of the update steps together with the update steps themselves.
	%
\end{proof}
\noindent
The relations in \Cref{lem:4} capture structural properties of the iterates
that will be used repeatedly in our later analysis,
particularly in the convergence arguments of Section~\ref{sec: Algorithm Implementation For the General Case}.

It is established by \cite[Proposition 1]{bolte2018nonconvex}  that any critical point of  \eqref{eq:25} is a critical point of \eqref{eq:26}, that is, $(x,y,z) \in \crit{\mathcal{L}_\beta} \Rightarrow x \in \crit{\{h(x)+P(\mathbb{E}[M]x)\}}$.
However, this implication cannot be applied directly since $\mathbb{E}[M]$ is unknown throughout the optimization process -- this poses a significant challenge which we tackle in the analysis of ISAD by utilizing its different elements.

Another challenge common in nonconvex AL schemes is  in controlling the multiplier sequence to generate a descent property.
We achieve this by employing an adaptive penalty parameter for \eqref{eq:25}.


\subsection{Penalty adaptation}
\label{sec:Lipschitz and Empirical Lipschitz Constants}

The penalty parameter in the AL function \eqref{eq:25} plays a crucial role in controlling the multiplier variable, here denoted by $z$. 
Indeed, the literature so far provides two strategies to bound the multiplier variable in the nonconvex AL setup. 
The first is assuming that the penalty parameter is chosen correctly using the problem's data and exploiting structural properties of the objective function  \cite{LP15,BN20}. 
The second is to utilize a penalty adaptation procedure  during the optimization process \cite{bolte2018nonconvex,hallak2023adaptive}.

Due to our use of an online estimator in place of the exact expectation, we employ nontrivial penalty adaptation procedures tailored to the given model. These procedures rely on the notion of \textit{empirical Lipschitz constants}. 

\begin{defin}[Empirical Lipschitz constant] \label{empirical_lipschitz_defin}  
Let $F : \mathbb{R}^k \to \mathbb{R}^l$, and let $\{x_t\}_{t \ge 0} \subseteq \mathbb{R}^k$.  
The empirical Lipschitz constant (eLip) of $F$ with respect to the sequence $\{x_t\}_{t \ge 0}$ is defined as  
\begin{equation*}
L^e_F (\{x_t\}_{t \ge 0}) =
\begin{cases}
\sup\limits_{t \ge 0,\, x_{t+1} \neq x_t} 
\dfrac{\|F(x_{t+1}) - F(x_t)\|}{\|x_{t+1} - x_t\|}, & \exists\, t \text{ with } x_{t+1} \neq x_t, \\
0, & \text{otherwise}.
\end{cases}
\end{equation*}
\end{defin}  

For comparison, the (global) Lipschitz constant of $F$ is defined by  
\begin{equation*}
L_F = \sup_{x \neq y} \frac{\|F(x) - F(y)\|}{\|x-y\|}.
\end{equation*} 

The empirical Lipschitz constant has three advantages:  
(i) it may exist even when a global Lipschitz constant does not;  
(ii) it can be computed online during the optimization process; and  
(iii) when a global constant exists, it can be much smaller, leading to less conservative penalty updates.  

\begin{lem} \label{lemma_empirical_lipschitz_bounded_by_lipschitz}
Let $F : \mathbb{R}^k \to \mathbb{R}^l$. 
For any sequence $\{x_t\}_{t \ge 0}$, it holds that  
\[
L^e_F(\{x_t\}_{t \ge 0}) \;\le\; L_F .
\]
\end{lem}  

\begin{lem}[Running estimate] \label{lemma_empirical_lipschitz_running_estimate}  
Let $F : \mathbb{R}^k \to \mathbb{R}^l$, and let $\{x_t\}_{t \ge 0} \subseteq \mathbb{R}^k$. Define the sequence $\{l_t\}$ by  
\[
l_0 = 0, \quad
l_{t+1} =
\begin{cases}
l_t, & x_{t+1} = x_t, \\
\max\!\left\{l_t, \dfrac{\|F(x_{t+1}) - F(x_t)\|}{\|x_{t+1} - x_t\|}\right\}, & \text{otherwise}.
\end{cases}
\]  
Then $\{l_t\}$ is nondecreasing and  
\[
\lim_{t \to \infty} l_t \;=\; L^e_F(\{x_t\}_{t \ge 0}).  
\]  
\end{lem}  

This lemma shows how empirical Lipschitz constants can be tracked in practice: by maintaining a simple running maximum of observed local slopes. Along any bounded trajectory, the resulting sequence of estimates is finite, since slopes are taken over a compact set. The empirical Lipschitz constant plays a central role in the design and analysis of our \textit{Adaptive Penalty Oracles (APOs)}. The APO requirements are outlined in \Cref{adaptive_penalty_oracle_requirements_defin}, with two concrete implementations provided in \Cref{sec: Algorithm Implementation For the General Case}.

\section{Iterative Sampling Alternating Directions Method} \label{sec: Iterative Sampling Alternating Direction Method of Multipliers}
The \textit{Iterative Sampling Alternating Directions (ISAD)} method in its meta-algorithm form using the ingredients from \Cref{sec:Mathematical Ingredients} is given in \Cref{alg:1}.


The meta-algorithm provides two degrees of freedom to its implementation: the choice of the convex function $\phi$, and the implementation of the APO. The required characteristics of the APO are introduced in \Cref{adaptive_penalty_oracle_requirements_defin}. 
We discuss its general implementation in \cref{sec: Algorithm Implementation For the General Case} and propose a simpler implementation for when  $\nabla^2 h(x)$ has known lower and upper bounds in \cref{sec: A refined method to bounded hessian problems}.

\Cref{alg:1} uses several parameters and notation. 
The sequence $\{\theta_t\}_{t\ge0}$ determines the number of samples $\{M^i\} $ of the matrix $M$ that are collected by round $t$ -- at round $t+1$, we sample $\theta_{t+1} - \theta_t$ matrices. 
The decision variables of the AL function are $(x,y,z)$. 
The function $\phi$ is the basis for the Bregman distance factor $D_\phi$ which is utilized in the update of $x$.
Finally, $\{ \beta_t \}$ is the sequence of  penalties updated adaptively using the APO.

\noindent	\begin{algorithm2e}[htbp]
	\caption{Meta-ISAD}\label{alg:1}
	\KwIn{$\bx_0\in\real^n$, $\bz_0\in\real^n$, $\beta_0>0, \{\theta_t\}_{t\ge0}$, $\phi: \real^n \rightarrow \real$ where $\phi\in C^2$ is convex.} 
	
	\For{$t=0,1,2, \ldots$}{
		sample $(\theta_{t+1}-\theta_t)$ matrices $\{M^i\}_{i=\theta_t + 1}^{\theta_{t+1}}$\;
		$\bar{M}^{t+1} = \dfrac{\theta_t}{\theta_{t+1}}\bar{M}^{t} + \dfrac{1}{\theta_{t+1}} \sum\limits_{i=\theta_t + 1}^{\theta_{t+1}} M^i\;$
		
		set $(y_{t+1},x_{t+1},z_{t+1})\leftarrow \mathcal{A}(x_{t},z_{t}, \bar{M}^{t+1}, \beta_{t}, \phi )$\;
		
%
%
		
		$\beta_{t+1} \leftarrow APO(\bar{M}^{t+1}, \beta_{t})$ \tcp{input depends on setup}
	}
\end{algorithm2e}


Throughout our analysis we will assume that the generated sequence is bounded.
This assumption is common in the analysis of AL-based methods applied to nonconvex nonsmooth problems -- see \cite{bolte2018nonconvex,BST13,BN20,CHT2021,hallak2023adaptive}. 
In \cite{LP15,YPC17}, the boundedness of the sequence is not assumed, but coerciveness that implies boundedness in their studied model is assumed instead.

\begin{assumption}\label{assum:2} 
The sequence $\{x_t, y_t, z_t\}_{t\ge0}$ generated by \Cref{alg:1} is bounded.
\end{assumption}

\begin{remark}[On boundedness of the iterates]\label{rem:boundedness}
In the deterministic setting, boundedness of the iterates follows from
coerciveness assumptions on $h$ or $P$ once a uniform upper bound is
available for a Lyapunov function built from the augmented Lagrangian.
In \cite[Theorem~3]{LP15}, that upper bound is a consequence of the
monotone decrease of this Lyapunov function.
This route is not directly available in our setting. Since $\mathbb{E}[M]$ is replaced
by the estimator $\bar{M}^t$, the augmented Lagrangian decreases only up to
error terms induced by the estimation errors $\delta_t$ (cf.\ Milestone~3 in
the proof of \Cref{thm:3}), and these error terms scale
quadratically with the norms of the iterates. Consequently, their
summability, and with it a uniform upper bound on the augmented Lagrangian,
is itself established using the boundedness of the iterates. Deriving
verifiable sufficient conditions for \Cref{assum:2}
in the stochastic setting is left for future research.
\end{remark}


We define two functions that will be used extensively during the analysis of the algorithm.
The first is the AL function with respect to a matrix $A$, which we denote by $\mathcal{L}_{\beta}(x, y, z; A)$:
\begin{equation*}
    \mathcal{L}_{\beta}(x, y, z; A) := h(x) + P(y) - \langle z, A x - y \rangle + \dfrac{\beta}{2}\|A x - y\|^2.
\end{equation*}
The second is  the function whose critical points we are searching for during the $x$ update:
\begin{equation}\label{g_function_defin}
    g^{t+1}(x) = h(x) - \langle z_t, \bar{M}^{t+1} x \rangle + \dfrac{\beta_t}{2}\| \bar{M}^{t+1}x - y_{t+1}\|^2 + D_{\phi}(x,x_t).
\end{equation}
Note that $g^{t+1}(x)$ is equal to $\mathcal{L}_{\beta_t}(x, y_{t+1}, z_t; \bar{M}^{t+1}) + D_\phi(x,x_t)$ up to additive constants that do not depend on $x$.

With  \eqref{g_function_defin} at hand, we can define the APO function in \Cref{alg:1}.
\begin{defin}[APO]\label{adaptive_penalty_oracle_requirements_defin}
An APO is a function that provides a positive scalar penalty parameter $\beta > 0$. When used in \Cref{alg:1}, it ensures that for any sequence of values $\{x_t, y_t, z_t\}_{t \ge 0}$ produced by the algorithm, and the corresponding functions $\{g^t\}_{t \geq 1}$ (as defined in \eqref{g_function_defin}), 
there exists with probability $1$ a finite random variable $K_{stable} > 0$, such that the following conditions hold:
\begin{enumerate}
    \item \textbf{Stability:} For all $t > K_{stable}$,  $\beta_t = \beta_{K_{stable}}$.
 
    \item \textbf{Sufficient Decrease:} There exists $\rho > 0$ such that for all $t > K_{stable}$ it holds that $g^{t+1}(x_{t+1}) - g^{t+1}(x_t) \leq - \rho \|x_{t+1} - x_t\|^2 / 2, $
    where $\rho$ satisfies 
    \begin{align*}
        \rho &> \dfrac{8}{\beta_t \sigma} (\mu^2 + \nu^2),
    \end{align*}
    with $\mu := L^e_{\nabla h + \nabla \phi}(\{x_t\}_{t \geq 0})$ and $\nu := L^e_{\nabla \phi}(\{x_t\}_{t \geq 0})$ the empirical Lipschitz constants defined in \Cref{empirical_lipschitz_defin}.
\end{enumerate}
\end{defin}

Under \Cref{assum:2}, $\mu$ and $\nu$ are finite: since $h$ and $\phi$ are twice continuously differentiable, the maps $\nabla h + \nabla \phi$ and $\nabla \phi$ are Lipschitz continuous on any compact set containing the iterates (see the proof of \Cref{lem:ifstatement_rhs_is_bounded}, in particular \eqref{eq:ifstatement_rhs_is_bounded:zeta_bound} and \eqref{eq:ifstatement_rhs_is_bounded:xi_bound}). Finiteness also holds independently of \Cref{assum:2} whenever $\nabla h$ and $\nabla \phi$ are globally Lipschitz continuous (\Cref{lemma_empirical_lipschitz_bounded_by_lipschitz}), for instance in the setting of \Cref{sec: A refined method to bounded hessian problems}.

\begin{remark}[On stabilization of penalty updates]
One of the defining properties of an Adaptive Penalty Oracle (APO) is that 
the sequence of penalty parameters $\{\beta_t\}$ stabilizes after finitely many updates. 
This requirement is part of the APO specification: any procedure failing to ensure 
finite stabilization would not qualify as an APO. 
Intuitively, such stabilization is natural in our setting: 
since the iterates remain bounded and empirical Lipschitz estimates are finite, 
the sufficient-decrease condition used in the APO can eventually be satisfied 
with a sufficiently large penalty parameter.
\end{remark}

\begin{remark}[APO existence]
    An APO that satisfies \cref{adaptive_penalty_oracle_requirements_defin} exists -- \Cref{general_beta_oracle} implements such an oracle for the general case, while in \Cref{sec: A refined method to bounded hessian problems} we provide a specific simple implementation for when $\nabla^2 h$ is bounded.
\end{remark}


\section{The Sampling Mechanism} 
\label{sec: Properties of the Sampling Mechanism}
This section presents the properties of the sampling mechanism and its implications, which form the backbone of our analysis of the meta-algorithm and its instances. 
Denote the \textit{error of the matrix estimator with respect to $\bar{M}$} by 
\begin{equation}
	\label{eq:error matrix general}
	\delta (\bar{M}, M) := \bar{M} - \mathbb{E}[M],
\end{equation}
where $\bar{M}$ is the estimator of $\mathbb{E}[M]$.
The error term in \eqref{eq:error matrix general} will almost always be given with respect to a sequence sampled by \Cref{alg:1}, and therefore, we will use the abbreviation  
\begin{equation}
\label{eq:error matrix}
    \delta_t := \delta (\bar{M}^t, M) = \bar{M}^t - \mathbb{E}[M],
\end{equation}
where we assume that the couple $(\bar{M}^t, M)$ is known from context; note that $\delta_t$ is a random variable until realized.

To establish the essential properties of the random variable $\delta_t$, we will utilize classical elements from stochastic analysis, including the General Hoeffding's inequality \cite[Theorem 2.6.2]{vershynin_2018} and the Borel--Cantelli Theorem \cite[Theorem 11.1.1]{borovkov2013probability}, both of which are used without stating them, and the notion of \textit{infinitely often (i.o.)} occurring event, which is defined below. 
\begin{defin}[infinitely often (i.o.)]
Let $A_1, A_2, ...$ be a sequence of events in some probability space. Then the probability that infinitely many of them occur is denoted by $\mathbb{P} \left( A_t \; i.o. \right) \equiv \mathbb{P} \left(\limsup_{t\to\infty} A_t \right) \equiv \mathbb{P} \left(\bigcap\limits_{t=1}^\infty \bigcup\limits_{k=t}^{\infty} A_k \right). $
\end{defin}
We will use $i.o.$ occurring events  to prove that $\mathbb{P} \left(\sum\limits_{t=1}^\infty \|\delta_t\|^2 < \infty \right) = 1$.
\begin{lem}\label{lem:not_io_then_bounded}
    Let $\{Y_t\}_{t \ge 0}$ be a sequence of random variables. 
	If there exist $\epsilon > 0$ and $\alpha > 0$ such that $\mathbb{P} \left(\| Y_t\| > \alpha t^{-(0.5 + 0.25\cdot\epsilon)} \; i.o. \right) = 0 $,
    then 
    \begin{equation*}
        \mathbb{P} \left(\sum\limits_{t=1}^\infty \|Y_t\|^2 < \infty \right) = 1. 
    \end{equation*}
\end{lem}
\begin{proof}
	Since $\mathbb{P} \left(\| Y_t\| > \alpha t^{-(0.5 + 0.25\cdot\epsilon)} \; i.o. \right) = 0 $, there exists with probability $1$ a finite index $K$, such that for all $t \ge K+1$, $\left\Vert Y_t\right\Vert \le \alpha t^{-(0.5 + 0.25\cdot\epsilon)}$. 
	Since the image of $\|Y_t\|$ is contained in $\real$, it follows that $\sum\limits_{t=1}^K \|Y_t\|^2 < \infty$. Since for every $t \ge K+1$, $\|Y_t\| \le \alpha t^{-(0.5 + 0.25\cdot\epsilon)}$, it follows that $\sum\limits_{t=1}^\infty \|Y_t\|^2 = \sum\limits_{t=1}^K \|Y_t\|^2 + \sum\limits_{t=K+1}^\infty \|Y_t\|^2 < \infty. $
\end{proof}
\cref{lem:not_io_then_bounded} suggests that we should pick $\{\theta_t\}_{t\ge0}$ so that 
\begin{equation}
	\label{eq:949}
	\mathbb{P} \left(\|\delta_t\| > \alpha t^{-(0.5 + 0.25\cdot\epsilon)} \; i.o. \right) = 0
\end{equation}
for some $\alpha > 0$.
%
%
%
%
As we stated previously, the sampling rate regime, expressed by the sequence $\{ \theta_t\}_{t \ge 0}$, is  indeed chosen to guarantee that \eqref{eq:949} holds true.

In the general case, $\theta_t$ should be proportional to $t^{2+\epsilon}$ -- this is described by \Cref{lem:1}.
When the elements of the sampled matrix are sub-Gaussian, we derive a milder sampling rate proportional to $t^{1+\epsilon}$  -- as stated in \Cref{lem:2}.
Both lemmas are based on the Borel--Cantelli Theorem. 
\Cref{lem:1} additionally uses Chebyshev's inequality, while \Cref{lem:2} uses the General Hoeffding's Inequality.
Due to the technicality  of the proofs, we provide them in the appendix.

\begin{lem}\label{lem:1}
Let $\{M^t\}_{t\ge0}$ be a sequence of i.i.d. random matrices with distribution $M^i \sim \pi$ sampled by \Cref{alg:1} with $\epsilon > 0$ and $\theta_t = t^{2+\epsilon}$. 
Then \eqref{eq:949} holds true.
\end{lem}

\begin{lem}\label{lem:2} 
Let $\{M^t\}_{t\ge0}$ be a sequence of i.i.d. random matrices with distribution $M^i \sim \pi$ sampled during the execution of \Cref{alg:1} with $\theta_t = t^{1+\epsilon}$ for some $\epsilon > 0$. 
Assume that for every $i,j \in [n]$ and every $t \ge 0$, $M^t_{i,j}$ is sub-Gaussian.
Then \eqref{eq:949} holds true.
\end{lem}

Note that  $\epsilon$ can be chosen to be arbitrarily small. 
If  $\{M^i\}_{i=1}^\infty$ are of a sub-Gaussian distribution as stated in the conditions of $\Cref{lem:2}$, we can choose $\theta_t = t^{1+\epsilon}$ so that $\theta_t$ is arbitrarily close to being linear in $t$.

The next claim states a standard technical result on the  convergence of the eigenvalues $\lambda_{\min}((\bar{M^t})^T \bar{M^t})$ to $\lambda_{\min}(\mathbb{E}[M]^T \mathbb{E}[M])$; a proof is given in \cite{greenstein2023augmented}.

\begin{lem}\label{lem:3}
Let $\{\bar{M}^t\}_{t\ge0}$ be a sequence of unbiased estimators of $\mathbb{E}[M]$, such that $\bar{M}^t \xrightarrow{t\to\infty} \mathbb{E}[M]$ almost surely. 
For any $\epsilon \in (0, 1)$, with probability 1, there exists $K>0$ such that for any $t>K$ it holds that $ (1-\epsilon)\sigma \le \lambda_{\min}((\bar{M^t})^T \bar{M^t}) \le (1 + \epsilon) \sigma $ and $(1-\epsilon)\lambda_{\min}((\bar{M^t})^T \bar{M^t}) \le \sigma \le (1 + \epsilon) \lambda_{\min}((\bar{M^t})^T \bar{M^t}) $,
where $\sigma = \lambda_{\min}(\mathbb{E}[M]^T \mathbb{E}[M])$, as defined in \Cref{assum:blanket}. 
\end{lem}


\section{Meta Algorithm Implementation and Analysis} \label{sec: Algorithm Implementation For the General Case}
\label{sec: Meta Algorithm Convergence}
The meta algorithm uses the general APO described in  \Cref{general_beta_oracle}, where \cref{general_beta_oracle_fulfills_requirements}  establishes that 
it  indeed satisfies the requirements outlined in \cref{adaptive_penalty_oracle_requirements_defin} when it is invoked by the meta-algorithm. 

\noindent	\begin{algorithm2e}[htbp]
	\caption{General APO}\label{general_beta_oracle}
	\KwIn{$\bx_{t+1}\in\real^n$, $\bx_{t}\in\real^n$, $\beta_t>0, \zeta_t, \xi_t$, $\epsilon > 0$, $\bar{M}^{t+1}$, $\phi$, $h$, where $\bar{M}^{t+1}$, $\phi$ and $h$ satisfy required assumptions.} 

        \uIf{$x_{t+1} = x_t$}{
            return $\beta_t, \zeta_t, \xi_t$\;
        }
            
        $\zeta_{t+1} = \max \left\{\zeta_t, \dfrac{\|\nabla h(x_{t+1}) + \nabla \phi(x_{t+1}) - \nabla h(x_t) - \nabla \phi(x_t)\|^2}{\|x_{t+1}-x_t\|^2}\right\}$\;
		
        $\xi_{t+1} = \max \left\{ \xi_t, \dfrac{\|\nabla \phi(x_{t+1}) - \nabla \phi(x_t)\|^2}{\|x_{t+1}-x_t\|^2} \right\}$\;

        Set $\tilde{\sigma}_{t+1} = \lambda_{\min}((\bar{M}^{t+1})^T \bar{M}^{t+1})$\;

        \uIf{$\tilde{\sigma}_{t+1} = 0$}{
            return $\beta_t, \zeta_{t+1}, \xi_{t+1}$\;
        }
  
     Set $\rho_{t+1} = -\dfrac{2\left(g^{t+1}(x_{t+1}) - g^{t+1}(x_t)\right)}{\|x_{t+1}-x_t\|^2}$\;
		
		\uIf{$\dfrac{1}{4} \rho_{t+1} > \dfrac{8(\zeta_{t+1} + \xi_{t+1} + \epsilon)}{\beta_t  \tilde{\sigma}_{t+1}}$}
		{return $\beta_t, \zeta_{t+1}, \xi_{t+1}$\;}
		\Else {return $2 \beta_t, \zeta_{t+1}, \xi_{t+1}$\;}
	
\end{algorithm2e}

\begin{thm}\label{general_beta_oracle_fulfills_requirements}
Let $\{(x_t, y_t, z_t)\}_{t\ge0}$ be the sequence generated by \Cref{alg:1} with the APO defined in \Cref{general_beta_oracle}, and let $\{\beta_t\}_{t\ge0}$ be the sequence of adaptive penalties generated. Then $\{x_t, y_t, z_t, \beta_t\}_{t\ge0}$ satisfies \cref{adaptive_penalty_oracle_requirements_defin}. 
\end{thm}
Two auxiliary lemmas are utilized in the proof of \cref{general_beta_oracle_fulfills_requirements}. 
\begin{lem}[auxiliary 1]\label{lem:ifstatement_rhs_is_bounded}
    Let $\{(x_t, y_t, z_t, \beta_t)\}_{t\ge0}$ be the sequence generated by \Cref{alg:1} with the APO defined in \Cref{general_beta_oracle}, 
	let $\tilde{\sigma}_t = \lambda_{min}\left(\left(\bar{M}^t\right)^T \bar{M}^t\right)$ and let $\epsilon > 0$ be the positive constant used in \Cref{general_beta_oracle}. 
	Then there almost surely exists an index $J \ge 0$ and a constant $C\in(0,\infty)$ such that for any $t > J$, it holds that $\tilde{\sigma}_t > 0$ and
    \begin{equation*}
        \dfrac{8(\zeta_{t+1} + \xi_{t+1} + \epsilon)}{\beta_t  \tilde{\sigma}_t} < C, \qquad \forall t>J.
    \end{equation*}
\end{lem}
\begin{proof}
We bound each component separately. 
Since \( \{\beta_t\}_{t \geq 0} \) is nondecreasing, it follows that $ \beta_0 \leq \beta_t $ for any $t\geq 0$, and by \cref{lem:3}, with probability 1, there exists an index \( J \geq 0 \) such that $ \tilde{\sigma}_t \geq 0.5 \sigma$ for any $t > J $.

From \cref{empirical_lipschitz_defin} and \cref{lemma_empirical_lipschitz_running_estimate}, we have that
\[ \mu^2 = \sup_{t \geq 0} \zeta_t, \quad \nu^2 = \sup_{t \geq 0} \xi_t. \]
Assume that there exists an index $t$ such that $x_{t+1} \ne x_t$. By \cref{assum:2}, there exists \( D > 0 \) such that \( \max_{t \geq 1} \{\|x_t\|, \|y_t\|, \|z_t\|\} \leq D \), yielding
\begin{align*}
    \zeta_t &\leq \sup_{t \geq 0} \zeta_t = \mu^2 \\
    &=\sup\limits_{t \ge 0, x_{t+1} \ne x_t} \dfrac{\|\nabla h(x_{t+1}) + \nabla \phi(x_{t+1}) - \nabla h(x_t) - \nabla \phi(x_t)\|^2}{\|x_{t+1} - x_t\|^2} \\&\le \sup_{x,y \in \mathcal{B}(0,D), x \neq y} \frac{\|\nabla h(x) + \nabla \phi(x) - \nabla h(y) - \nabla \phi(y)\|^2}{\|x - y\|^2} =: \left(L^D_{\nabla \phi + \nabla h}\right)^2.
\end{align*}
Similarly,
\begin{align*}
    \xi_t \leq \sup_{t \geq 0} \xi_t = \nu^2 
    &=\sup\limits_{t \ge 0, x_{t+1} \ne x_t} \dfrac{\|\nabla \phi(x_{t+1}) - \nabla \phi(x_t)\|^2}{\|x_{t+1} - x_t\|^2} \\&\le \sup_{x,y \in \mathcal{B}(0,D), x \neq y} \frac{\|\nabla \phi(x) - \nabla \phi(y)\|^2}{\|x - y\|^2} =: \left(L^D_{\nabla \phi}\right)^2.
\end{align*}

If there is no index $t$ such that $x_{t+1} \ne x_t$, then $\zeta_t = \xi_t = 0$, and the inequalities $\zeta_t \le \left(L^D_{\nabla \phi + \nabla h}\right)^2$ and $\xi_t \le \left(L^D_{\nabla \phi}\right)^2$ hold true trivially. Note that $L^D_{\nabla \phi + \nabla h}$ and $L^D_{\nabla \phi}$ are the Lipschitz constants of the restrictions of $\nabla \phi + \nabla h$ and $\nabla \phi$ respectively to the closed ball $\mathcal{B}(0,D)$. The continuously differentiable gradients have bounded Lipschitz constants over compact sets, since the eigenvalues of the Hessians are bounded; i.e., the values
\begin{align}\label{eq:ifstatement_rhs_is_bounded:lambda_h_phi}
    \lambda^{h+\phi}_{\min} &= \min_{x \in B(0,D)} \lambda_{\min}(\nabla^2 \phi(x) + \nabla^2 h(x)), \nonumber\\ \lambda^{h+\phi}_{\max} &= \max_{x \in B(0,D)} \lambda_{\max}(\nabla^2 \phi(x) + \nabla^2 h(x)), \\
    \lambda^{\phi}_{\min} &= \min_{x \in B(0,D)} \lambda_{\min}(\nabla^2 \phi(x)), \nonumber\\ \lambda^{\phi}_{\max} &= \max_{x \in B(0,D)} \lambda_{\max}(\nabla^2 \phi(x))\nonumber
\end{align}
exist and are finite according to Weierstrass' Extreme Value Theorem.
Since Lipschitz constants of gradients equal the maximum singular value of Hessians, it follows that
\begin{equation}\label{eq:ifstatement_rhs_is_bounded:zeta_bound}
    \zeta_t \leq L^D_{\nabla \phi + \nabla h} = \max \{ |\lambda^{h+\phi}_{\min}|, |\lambda^{h+\phi}_{\max}| \} < \infty,
\end{equation}
\begin{equation}\label{eq:ifstatement_rhs_is_bounded:xi_bound}
    \xi_t \leq L^D_{\nabla \phi} = \max \{ |\lambda^{\phi}_{\min}|, |\lambda^{\phi}_{\max}| \} < \infty.
\end{equation}

Combining the fact that $\beta_0 \leq \beta_t$, $ \tilde{\sigma}_t \geq 0.5 \sigma$, \eqref{eq:ifstatement_rhs_is_bounded:zeta_bound}, and \eqref{eq:ifstatement_rhs_is_bounded:xi_bound}, the result of the lemma follows with $C = 16 (\sigma \beta_0)^{-1} \left( \left(L^D_{\nabla \phi + \nabla h}\right)^2 + \left(L^D_{\nabla \phi}\right)^2 + \epsilon\right). $
\end{proof}


\begin{lem}[auxiliary 2]\label{lem:g_hessian_eigenvalue_is_lower_bounded}
    Let $\{(x_t, y_t, z_t,\beta_t)\}_{t\ge0}$ be the sequence generated by \Cref{alg:1} with the APO defined in \Cref{general_beta_oracle}, and let $J$ and $C$ be defined as in \cref{lem:ifstatement_rhs_is_bounded}. 
	Then, there exists a constant $\kappa \ge 0$, such that one of the following holds true:
    \begin{enumerate}
        \item  $\beta$ is updated fewer than $\kappa$ times starting from iteration $J+1$.
        \item  $\beta$ is updated at least $\kappa$ times starting from iteration $J+1$, and let $\tau$ be the iteration of the $\kappa$ update (if $\kappa = 0$, then $\tau = J+1$). Then, 
        \begin{equation*}
            \min\limits_{x \in \mathcal{B}(0,D)} \lambda_{\min}\left( \nabla^2 g^t(x) \right) > 4C \qquad \forall t > \tau,
        \end{equation*}
        where $D > 0$ is a constant satisfying \( \max_{t \geq 1} \{\|x_t\|, \|y_t\|, \|z_t\|\} \leq D \), whose existence is assumed in \cref{assum:2}, and $\mathcal{B}(0,D)$ is a closed ball.
    \end{enumerate}
\end{lem}
\begin{proof}
Define $\kappa$ as  
\[
\kappa = \begin{cases}
0, & \beta_J \ge \dfrac{5C - \lambda^{h+\phi}_{\min}}{0.5 \sigma},\\
\left\lceil \log_2 \left(\dfrac{5C - \lambda^{h+\phi}_{\min}}{0.5 \beta_J \sigma} \right) \right\rceil, & \text{otherwise},
\end{cases}
\]
where $\lambda^{h+\phi}_{\min}$ is given in \eqref{eq:ifstatement_rhs_is_bounded:lambda_h_phi}, and consider three cases:  
(i) $\kappa = 0$;
(ii) $\kappa > 0$ and $\beta$ is updated fewer than $\kappa$ times; and  
(iii) $\kappa > 0$ and $\beta$ is updated at least $\kappa$ times.  

In any case, for any $t$, we have  $\lambda_{\min} \left(\nabla^2 g^t(x) \right) \geq \lambda_{\min} \left( \nabla^2 h(x) + \nabla^2 \phi(x) \right) + \beta_t \lambda_{\min} \left( \left(\bar{M}^t\right)^T \bar{M}^t\right) $, and 
thus, using the definition of $\lambda^{h+\phi}_{\min}$ from \eqref{eq:ifstatement_rhs_is_bounded:lambda_h_phi}, it follows that  
\[
\min\limits_{x \in \mathcal{B}(0,D)} \lambda_{\min}\left( \nabla^2 g^t(x) \right) \geq  \lambda^{h+\phi}_{\min} + \beta_t \lambda_{\min} \left( \left(\bar{M}^t\right)^T \bar{M}^t\right).
\]

By our choice of $J$ in \cref{lem:ifstatement_rhs_is_bounded}, for any $t > J$,  $\lambda_{\min} \left( \left(\bar{M}^t\right)^T \bar{M}^t\right) > 0.5 \sigma. $
Consequently, for $t > J$, we obtain  
\begin{equation}\label{eq:g_hessian_eigenvalue_is_lower_bounded:general_g_eigenvalue_lower_bound}
\min\limits_{x \in \mathcal{B}(0,D)} \lambda_{\min}\left( \nabla^2 g^t(x) \right) \geq  \lambda^{h+\phi}_{\min} + 0.5 \beta_t \sigma.
\end{equation}
We now analyze each case.

Suppose that $\kappa = 0$.
Then we have $\tau = J+1$ by the definition of $\tau$ for $\kappa = 0$, and $\beta_J \geq \frac{5C - \lambda^{h+\phi}_{\min}}{0.5 \sigma}$ by the construction of $\kappa$.  
Since $\{\beta_t\}_{t \geq 0}$ is nondecreasing, it follows that $\beta_t \geq \beta_J$ for all $t > \tau = J+1$.  
Therefore, we derive from \eqref{eq:g_hessian_eigenvalue_is_lower_bounded:general_g_eigenvalue_lower_bound} that
\[
\min\limits_{x \in \mathcal{B}(0,D)} \lambda_{\min}\left( \nabla^2 g^t(x) \right) \geq  \lambda^{h+\phi}_{\min} + 0.5 \beta_{J} \sigma \geq  5C > 4C,
\]
which concludes the correctness of the claim in the first case.

Now, if $\kappa > 0$ and $\beta$ is updated fewer than $\kappa$ times then the lemma holds trivially, concluding the second case.

Finally, suppose that $\kappa > 0$ and $\beta$ is updated at least $\kappa$ times.  
Then by the definition of $\kappa$, we have $0 < \beta_0 \leq \beta_J < \dfrac{5C - \lambda^{h+\phi}_{\min}}{0.5 \sigma}$.  
This implies that $\log_2 \left(\dfrac{5C - \lambda^{h+\phi}_{\min}}{0.5 \beta_{\tau + 1} \sigma}\right)$ is well-defined.  
Since $\beta$ doubles at each update (cf. \Cref{general_beta_oracle}), we apply the definitions of $\tau$ and $\kappa$ to obtain  
\begin{align*}
\log_2 \left(\dfrac{5C - \lambda^{h+\phi}_{\min}}{0.5 \beta_{\tau+1} \sigma}\right) = \log_2 \left(\dfrac{5C - \lambda^{h+\phi}_{\min}}{2^\kappa \left( 0.5 \beta_{J} \sigma\right)}\right) 
&= \log_2 \left(\dfrac{5C - \lambda^{h+\phi}_{\min}}{ 0.5 \beta_{J} \sigma}\right) - \log_2\left(2^\kappa\right) \\
&= \log_2 \left(\dfrac{5C - \lambda^{h+\phi}_{\min}}{ 0.5 \beta_{J} \sigma}\right) - \kappa \leq 0.
\end{align*}
Thus 
\begin{equation*}
    \dfrac{5C - \lambda^{h+\phi}_{\min}}{0.5 \beta_{t} \sigma} \le 1 \quad \forall t > \tau,
\end{equation*}
which implies
\begin{equation*}
    0.5 \beta_{t} \sigma + \lambda^{h+\phi}_{\min} \ge 5C > 4C \quad \forall t > \tau.
\end{equation*}
Combining with \eqref{eq:g_hessian_eigenvalue_is_lower_bounded:general_g_eigenvalue_lower_bound}, this completes the proof of the lemma.
\end{proof}
Using \cref{lem:ifstatement_rhs_is_bounded} and \cref{lem:g_hessian_eigenvalue_is_lower_bounded}, we prove \cref{general_beta_oracle_fulfills_requirements}.
\begin{proof}[Proof of {\cref{general_beta_oracle_fulfills_requirements}}]

Our goal is to show that, with probability 1, there exists a random variable \( K_{stable} > 0 \) such that for all \( t > K_{stable} \), we have \( \beta_t = \beta_{K_{stable}} \) and some \( \rho > 0 \) satisfying  
\begin{align*}
		&g^{t+1}(x_{t+1}) - g^{t+1}(x_t) \le -\dfrac{\rho}{2}\|x_{t+1}-x_t\|^2, \text{ where} \\
		&\rho > \dfrac{8}{\beta_t \sigma}(\mu^2 + \nu^2).
\end{align*}

We use \cref{lem:ifstatement_rhs_is_bounded} and \cref{lem:g_hessian_eigenvalue_is_lower_bounded} to prove that the number of $\beta$ updates is finite with probability 1. Since \( J \) is finite with probability 1 (\Cref{lem:ifstatement_rhs_is_bounded}), the number of \(\beta\) updates up to iteration \( J \) is finite. If fewer than \(\kappa\) updates occur after \( J \), then the number of $\beta$ updates is finite. Otherwise, by \cref{lem:g_hessian_eigenvalue_is_lower_bounded} there exists \(\tau > J\) such that  
\[
\alpha_t := \min_{x \in \mathcal{B}(0,D)} \lambda_{\min}\left( \nabla^2 g^t(x) \right) > 4C, \quad \forall t > \tau.
\]

Since \( C > 0 \) (cf. \Cref{lem:ifstatement_rhs_is_bounded}), it follows that \( g^{t+1}(x) \) is strongly convex over the closed ball \( \mathcal{B}(0,D) \) with strong convexity parameter \( \alpha_{t+1} > 4C \). 
Thus,   
\begin{equation*}
    g^{t+1}(x_t) \ge g^{t+1}(x_{t+1}) + \langle \nabla g^{t+1}(x_{t+1}), x_t - x_{t+1} \rangle + \dfrac{\alpha_{t+1}}{2}\left\Vert x_{t+1} - x_t \right\Vert^2.
\end{equation*}
Since \( x_{t+1} \) is a critical point of \( g^{t+1}(x) \), we have \( \nabla g^{t+1}(x_{t+1}) = 0 \), implying  
\begin{equation*}
    g^{t+1}(x_t) \ge g^{t+1}(x_{t+1}) + \dfrac{\alpha_{t+1}}{2}\left\Vert x_{t+1} - x_t \right\Vert^2.
\end{equation*}
If \( x_{t+1} = x_t \), then \(\beta\) is not updated at iteration \( t+1 \). Otherwise,  
\begin{equation*}
    \alpha_{t+1} \le -\dfrac{2 \left(g^{t+1}(x_{t+1}) - g^{t+1}(x_t)\right)}{\left\Vert x_{t+1} - x_t \right\Vert^2} = \rho_{t+1}.
\end{equation*}
Since $\alpha_{t+1} > 4C$ and $\dfrac{8(\zeta_{t+1} + \xi_{t+1} + \epsilon)}{\beta_t  \tilde{\sigma}_t} < C$ for all $t > J$ (by \cref{lem:ifstatement_rhs_is_bounded}), it follows that 
\begin{equation*}
    \dfrac{1}{4} \rho_{t+1} > \dfrac{8(\zeta_{t+1} + \xi_{t+1} + \epsilon)}{\beta  \tilde{\sigma}} \qquad \forall t > \tau > J,
\end{equation*}
implying that the if condition in line $9$ of \Cref{general_beta_oracle} holds true, and $\beta$ is not updated in this case, either.

Thus, with probability 1, there exists \( K > 0 \) such that for all \( t > K \), \(\beta_t\) is not updated, proving that \Cref{general_beta_oracle} satisfies the first part of \Cref{adaptive_penalty_oracle_requirements_defin}.  
Accordingly we denote $\bar{\beta} := \beta_K$ which satisfies $\bar{\beta}  = \beta_t $ for all $t>K$.

Whenever $x_{t+1} = x_t$, the inequality 
\begin{equation}\label{eq:general_beta_oracle_fulfills_requirements:trivial_inequality}
    g^{t+1}(x_{t+1}) - g^{t+1}(x_t) \le -\dfrac{\rho}{2}\left\Vert x_{t+1}-x_t\right\Vert^2
\end{equation} 
trivially holds for any \( \rho \in \real \). 
Define  $\mathcal{I} = \{ t > K \mid x_{t+1} \neq x_t \}.$
If \( \mathcal{I} \) is empty or finite, then setting \( K_{stable} = K+1 \) or \( K_{stable} = \max_{t \in \mathcal{I}} t \),  \Cref{adaptive_penalty_oracle_requirements_defin} is satisfied with  
\begin{equation*}
    \rho = \dfrac{8(\mu^2 + \nu^2 + \epsilon)}{\bar{\beta} \sigma},
\end{equation*}
which is finite by \Cref{lem:ifstatement_rhs_is_bounded}. 

If \( \mathcal{I} \) is infinite, then by the construction of \( K \), we have \( K > J \) and  $1.5\sigma > \tilde{\sigma}_t > 0.5\sigma > 0 $ for any $t\in \mathcal{I}$.
Since $\tilde{\sigma}_t > 0$ and $x_{t+1} \ne x_t$, the conditions in lines 1 and 6 of \Cref{general_beta_oracle} are false for \( t \in \mathcal{I} \), implying that the condition $\dfrac{1}{4} \rho_{t+1} > 8(\zeta_{t+1} + \xi_{t+1} + \epsilon) (\bar{\beta}  \tilde{\sigma}_t)^{-1}$ in line 9 holds true.  
Using $\tilde{\sigma}_t < 1.5\sigma$ and rearranging,
\begin{equation}\label{eq:general_beta_oracle_fulfills_requirements:rho_t_inequality}
    \dfrac{3}{8} \rho_{t+1} - \dfrac{8(\zeta_{t+1} + \xi_{t+1} + \epsilon)}{\bar{\beta} \sigma} > 0 \qquad \forall t \in \mathcal{I}.
\end{equation}
Since $\lim_{t \to \infty} \zeta_t = \mu^2$ and $\lim_{t \to \infty} \xi_t = \nu^2$, taking $\liminf$ over \eqref{eq:general_beta_oracle_fulfills_requirements:rho_t_inequality} gives
\begin{equation*}
    \liminf\limits_{t \to\infty, t \in \mathcal{I}}\dfrac{3}{8} \rho_{t+1} \ge \dfrac{8(\mu^2 + \nu^2 + \epsilon)}{\bar{\beta} \sigma}.
\end{equation*}
Subsequently, by some simple mathematical manipulations we obtain that 
\begin{equation}\label{eq:general_beta_oracle_fulfills_requirements:rho_t_inequality_with_liminf_part_three}
    \liminf\limits_{t \to\infty, t \in \mathcal{I}} \rho_{t+1} > \dfrac{8(\mu^2 + \nu^2)}{\bar{\beta} \sigma}.
\end{equation}
Since \eqref{eq:general_beta_oracle_fulfills_requirements:rho_t_inequality_with_liminf_part_three} is a strict inequality, there exists $\epsilon'>0$, such that \eqref{eq:general_beta_oracle_fulfills_requirements:rho_t_inequality_with_liminf_part_three} holds true when we replace the left-hand side with $\rho := \left(1 - \epsilon'\right)\liminf\limits_{t \to\infty, t \in \mathcal{I}} \rho_{t+1}$.
Hence, 
\begin{equation}\label{eq:general_beta_oracle_fulfills_requirements:rho_bound}
    \rho > \dfrac{8(\mu^2 + \nu^2)}{\bar{\beta} \sigma}.
\end{equation}
By the definition of $\liminf$, there exists $\bar{K} > 0$, such that for all $\bar{K} < t \in \mathcal{I}$, $\rho_t > \left(1 - \epsilon'\right)\liminf\limits_{t \to\infty, t \in \mathcal{I}} \rho_{t+1} = \rho$. 
Set $K_{stable} = \max\left\{K, \bar{K}\right\}$. Then,
\begin{equation}\label{eq:general_beta_oracle_fulfills_requirements:update_bound_on_x_change}
    g^{t+1}(x_{t+1}) - g^{t+1}(x_t) < -\dfrac{\rho}{2}\left\Vert x_{t+1}-x_t\right\Vert^2 \qquad t \in \mathcal{I}, t > K_{stable}.
\end{equation}
For \( t > K_{stable}, t \notin \mathcal{I} \), we have \( x_{t+1} = x_t \), so \eqref{eq:general_beta_oracle_fulfills_requirements:trivial_inequality} holds.
Combining with \eqref{eq:general_beta_oracle_fulfills_requirements:update_bound_on_x_change} and \eqref{eq:general_beta_oracle_fulfills_requirements:rho_bound}, we conclude that \Cref{general_beta_oracle} satisfies the second part of \Cref{adaptive_penalty_oracle_requirements_defin}.  

\end{proof}


The proof of the  convergence guarantee of the meta  \Cref{alg:1} is split into two phases. 
First, in \Cref{thm:2} we prove that if the size of the algorithm steps converges to zero, meaning
\begin{equation*}
    \lim_{t\rightarrow\infty} \|x_{t+1}-x_t\|^2 + \|y_{t+1}-y_t\|^2 + \|z_{t+1}-z_t\|^2 = 0,
\end{equation*}
then every accumulation point of the algorithm is a critical point.
Then, in \Cref{thm:3} we prove that the size of the algorithm steps does indeed converge to zero.

Recall that we are using a subdifferential set defined as in \eqref{Subdifferential_def}.

\begin{thm}\label{thm:2}
Let $\{x_t, y_t, z_t\}_{t \ge 0}$ be the sequence generated by \Cref{alg:1}. Assume that $\lim_{t\rightarrow\infty} \|x_{t+1}-x_t\|^2 + \|y_{t+1}-y_t\|^2 + \|z_{t+1}-z_t\|^2 = 0$, that \cref{assum:blanket} and \cref{assum:2} hold true, and that $(x^*, y^*, z^*)$ is an accumulation point of the sequence. 
Then the following hold almost surely: (i) $-z^* \in \partial P(y^*)$; (ii) $\nabla h(x^*) = \mathbb{E}[M]^T z^*$; (iii) $\mathbb{E}[M]x^* = y^*$.
\end{thm}
\begin{proof}
By the definition of $\beta$ in \cref{adaptive_penalty_oracle_requirements_defin}, there exists an index $K_{stable}$, such that $\beta_t = \beta_{K_{stable}}$ for all $t > K_{stable}$. 
Let  $\{x_{t_i}, y_{t_i}, z_{t_i}\}_{i\geq 1}$ be a subsequence converging to the accumulation point $(x^*, y^*, z^*)$, and for convenience, set $\bar{\beta} = \beta_{K_{stable}}$ and assume without loss of generality that $t_1 > K_{stable}$, meaning that $\beta_{t_i} = \bar{\beta}$ for any $i\geq 1$.

The first part of the theorem requires a lengthy technical result to prove, and so we start with the second and the third parts.

For the second part, we note that by the second part of \Cref{lem:4}, for the $t_i$ element of the subsequence:
\begin{equation}\label{eq:28}
    \nabla h(x_{t_i}) - (\bar{M}^{t_i})^T z_{t_i} = - \nabla \phi (x_{t_i}) + \nabla \phi (x_{t_i - 1}).
\end{equation}

By the convergence of the subsequence $\{x_{t_i},y_{t_i}, z_{t_i}\}_{i\ge0}$ to $(x^*, y^*, z^*)$ and the assumption that $\lim_{t\rightarrow\infty} \|x_{t+1}-x_t\|^2 + \|y_{t+1}-y_t\|^2 + \|z_{t+1}-z_t\|^2 = 0$, it holds that $x_{t_i - 1} \xrightarrow{i\to\infty}x^*$.
Consequently, utilizing the fact that $\phi$ is continuously differentiable, it follows that $\nabla \phi(x_{t_i}), \nabla \phi(x_{t_i - 1}) \xrightarrow{i\to\infty} \nabla \phi(x^*)$,
and therefore, by taking $i\to\infty$ we  obtain that
\begin{equation}\label{eq:29}
    - \nabla \phi (x_{t_i}) + \nabla \phi (x_{t_i - 1}) \xrightarrow{i\to\infty} 0.
\end{equation}

Additionally, by the strong law of large numbers (see \cite[Theorem 11.3.1]{borovkov2013probability}), $ \bar{M}^{t_i} \xrightarrow{i\to\infty} \mathbb{E}[M]. $
Combining the latter with \eqref{eq:28} and \eqref{eq:29} yields the second part of the theorem, $ \nabla h(x^*) = \mathbb{E}[M]^T z^*. $

To see the correctness of the third part of the theorem, we note that by the third part of \Cref{lem:4} and the fact that  $\beta_{t_i}=\bar{\beta}$, for the $t_i$ element of the subsequence we have that $ \bar{M}^{t_i} x_{t_i} - y_{t_i} = \dfrac{1}{\beta_{t_i}} (z_{t_i-1} - z_{t_i})= \dfrac{1}{\bar{\beta}} (z_{t_i-1} - z_{t_i}). $
By our assumption that $\lim_{t\rightarrow\infty} \|z_{t+1}-z_t\|^2 = 0$, we thus obtain by taking $i\to\infty$ that $\mathbb{E}[M]x^* = y^*,$
which concludes the proof of the third part.

To prove the first part of the theorem, we will need to establish that 
\begin{equation}\label{P value convergence}
    \lim\limits_{i\to\infty} P(y_{t_i}) = P(y^*).
\end{equation}
We will a priori  assume that \eqref{P value convergence} holds, utilize it to prove the first part of the theorem, and only then establish that \eqref{P value convergence} indeed holds true.

Note that by the first part of \Cref{lem:4}, for the $t_i$ element of the subsequence, $0 \in \partial P(y_{t_i}) + z_{t_i} + \beta_{t_i} \bar{M}^{t_i} (x_{t_i} - x_{t_i-1}). $
Recall that for $t$ such that $t > K_{stable}$, $\beta_{t}=\bar{\beta}$, and that $t_1 > K_{stable}$.
Therefore, for all $i\geq 1$, 
\begin{equation*}
    0 \in \partial P(y_{t_i}) + z_{t_i} + \bar{\beta} \bar{M}^{t_i} (x_{t_i} - x_{t_i-1}).
\end{equation*}
Since $\lim\limits_{i\to\infty} y_{t_i} = y^*$ and $\lim\limits_{i\to\infty} P(y_{t_i}) = P(y^*)$ by our assumption, from \cite[Proposition 8.7]{rockafellar2009variational} we have that
\begin{equation*}
    \limsup\limits_{i\to\infty} \partial P(y_{t_i}) \subset \partial P(y^*).
\end{equation*}
Here, the notation
\[
\limsup_{i\to\infty} \partial P(y_{t_i})
\]
refers to the \textit{outer limit} of the sequence of sets 
$\{\partial P(y_{t_i})\}_{i\geq 1}$ in the sense of \cite[Definition~4.1]{rockafellar2009variational}.
Explicitly,
\begin{align*}
\limsup_{i\to\infty} &\partial P(y_{t_i})
:=
\\&\left\{
v \in \mathbb{R}^n \;|\;
\exists\, \text{subsequence } (i') \text{ and } v_{i'} \in \partial P(y_{t_{i'}})
\text{ such that } v_{i'} \to v
\right\}.
\end{align*}
That is, it consists of all vectors $v$ that can be obtained as limits of 
subgradients $v_{i'} \in \partial P(y_{t_{i'}})$ along some subsequence.

Taking $i\to\infty$, using the fact that $\limsup\limits_{i\to\infty} \partial P(y_{t_i}) \subset \partial P(y^*)$ and the assumption $\lim_{t\rightarrow\infty} \|x_{t+1}-x_t\|^2 = 0$, we conclude that $- z^* \in \partial P(y^*), $
which establishes the first part of the theorem -- given that \eqref{P value convergence} holds true.

All that remains is to prove the correctness of \eqref{P value convergence}.
To that end, it is sufficient to show that: (i) $\liminf\limits_{i\to\infty} P(y_{t_i}) \ge P(y^*)$, and (ii) $\limsup\limits_{i\to\infty} P(y_{t_i}) \le P(y^*)$, hold true.

The first relation 
\begin{equation}\label{P value liminf}
    \liminf\limits_{i\to\infty} P(y_{t_i}) \ge P(y^*)
\end{equation}
follows immediately from the lower semicontinuity of $P$.

By the convergence of the subsequence $\{x_{t_i}, y_{t_i}, z_{t_i}\}_{i\geq 0}$, and our assumption that $\lim_{t\rightarrow\infty} \|x_{t+1}-x_t\|^2 + \|y_{t+1}-y_t\|^2 + \|z_{t+1}-z_t\|^2 = 0$, it follows that 
\begin{equation*}
    x_{t_i - 1} \xrightarrow{i\to\infty} x^* \text{ and } z_{t_i - 1} \xrightarrow{i\to\infty} z^*.
\end{equation*}

To establish that $\limsup\limits_{i\to\infty} P(y_{t_i}) \le P(y^*)$ holds true,  
first note that since all the   components of $\mathcal{L}_{\bar{\beta}}(x,y,z; M)$ other than $P(y)$ are continuous and $\{(x_{t_i-1}, y_{t_i}, z_{t_i-1}, \bar{M}^{t_i})\} \xrightarrow{i\to\infty} (x^*, y^*, z^*, \mathbb{E}[M])$, it is sufficient to show that 
\begin{equation*}
    \limsup\limits_{i\to\infty}\mathcal{L}_{\bar{\beta}}(x_{t_i - 1}, y_{t_i},z_{t_i - 1}; \bar{M}^{t_i}) \le \mathcal{L}_{\bar{\beta}}(x^*, y^* ,z^*; \mathbb{E}[M]) := \mathcal{L}_{\bar{\beta}}(x^*, y^* ,z^*).
\end{equation*}

To prove the above, we will show that  $\{\mathcal{L}_{\bar{\beta}}(x_{t_i - 1}, y_{t_i},z_{t_i - 1}; \bar{M}^{t_i})\}_{i \ge 0}$ is bounded. 
Then, we will show that the limit of every convergent subsequence is bounded by $\mathcal{L}_{\bar{\beta}}(x^*, y^* ,z^*)$. 

Since $P$ is proper, there exists $\tilde{y}$ such that $P(\tilde{y}) < \infty$.
By \Cref{assum:2}, the sequence $\{x_t, y_t, z_t\}_{t \ge 0}$ is bounded, and as we stated previously, all the components of $\mathcal{L}_{\bar{\beta}}(x,y,z; M)$ except $P(y)$ are continuous. 
Subsequently, it follows that
\begin{equation*}
    \inf\limits_{i\geq 1} \{\mathcal{L}_{\bar{\beta}}(x_{t_i-1}, y_{t_i}, z_{t_i-1}; \bar{M}^{t_i}) - P(y_{t_i})\}
\end{equation*}
and
\begin{equation*}
    \inf\limits_{i\geq 1} \{\mathcal{L}_{\bar{\beta}}(x_{t_i-1}, \tilde{y}, z_{t_i-1}; \bar{M}^{t_i}) - P(\tilde{y})\},
\end{equation*}
\begin{equation*}
    \sup\limits_{i\geq 1} \{ \mathcal{L}_{\bar{\beta}}(x_{t_i-1}, y_{t_i}, z_{t_i-1}; \bar{M}^{t_i}) - P(y_{t_i})\}
\end{equation*}
and
\begin{equation*}
    \sup\limits_{i\geq 1} \{ \mathcal{L}_{\bar{\beta}}(x_{t_i-1},\tilde{y}, z_{t_i-1}; \bar{M}^{t_i}) - P(\tilde{y})\},
\end{equation*} 
exist and are finite. 

\noindent By definition of the update of $y$, $y_{t_i}$ is a minimizer of $\mathcal{L}_{\bar{\beta}}(x_{t_i - 1}, y, z_{t_i - 1}; \bar{M}^{t_i})$.
Hence, $  \mathcal{L}_{\bar{\beta}}(x_{t_i-1},y_{t_i},z_{t_i-1}) \le \mathcal{L}_{\bar{\beta}}(x_{t_i-1},\tilde{y},z_{t_i-1}).$
Moreover, by the definition of supremum and infimum,
\begin{align*}
    \mathcal{L}_{\bar{\beta}}(x_{t_i-1},y_{t_i},z_{t_i-1}) &= \mathcal{L}_{\bar{\beta}}(x_{t_i-1},y_{t_i},z_{t_i-1}) - P(y_{t_i}) + P(y_{t_i}) \\&\ge P(y_{t_i}) + \inf\limits_{j \ge 1} \{\mathcal{L}_{\bar{\beta}}(x_{t_j-1}, y_{t_j}, z_{t_j-1}; \bar{M}^{t_j}) - P(y_{t_j})\}
\end{align*}
and
\begin{align*}
    \mathcal{L}_{\bar{\beta}}(x_{t_i-1},\tilde{y},z_{t_i-1}) &= \mathcal{L}_{\bar{\beta}}(x_{t_i-1},\tilde{y},z_{t_i-1}) - P(\tilde{y}) + P(\tilde{y}) \\&\le P(\tilde{y}) + \sup\limits_{j \ge 1} \{\mathcal{L}_{\bar{\beta}}(x_{t_j-1}, \tilde{y}, z_{t_j-1}; \bar{M}^{t_j}) - P(\tilde{y}) \}.
\end{align*}
Combining the inequalities above, we get 
\begin{align*}
    P(y_{t_i}) &+ \inf\limits_{j \ge 1} \{\mathcal{L}_{\bar{\beta}}(x_{t_j-1}, y_{t_j}, z_{t_j-1}; \bar{M}^{t_j}) - P(y_{t_j})\} \\&\leq P(\tilde{y}) + \sup\limits_{j \ge 1} \{\mathcal{L}_{\bar{\beta}}(x_{t_j-1}, \tilde{y}, z_{t_j-1}; \bar{M}^{t_j}) - P(\tilde{y})\}. 
\end{align*}
Rearranging the terms, we obtain that for all $i\geq 1$ it holds that
\begin{align*}
    P(y_{t_i}) \leq P(\tilde{y}) &+ \sup\limits_{j \ge 1} \{\mathcal{L}_{\bar{\beta}}(x_{t_j-1}, \tilde{y}, z_{t_j-1}; \bar{M}^{t_j}) - P(\tilde{y})\} \\
    &- \inf\limits_{j\ge 1} \{\mathcal{L}_{\bar{\beta}}(x_{t_j-1}, y_{t_j}, z_{t_j-1}; \bar{M}^{t_j}) - P(y_{t_j})\}.
\end{align*}
Taking the limit, and using the fact that $ P(\tilde{y}) + \sup\limits_{j \ge 1} \{\mathcal{L}_{\bar{\beta}}(x_{t_j-1}, \tilde{y}, z_{t_j-1}; \bar{M}^{t_j}) - P(\tilde{y})\} 
 - \inf\limits_{j\ge 1} \{\mathcal{L}_{\bar{\beta}}(x_{t_j-1}, y_{t_j}, z_{t_j-1}; \bar{M}^{t_j}) - P(y_{t_j})\} $
is a sum comprising 3 finite elements, none of which depends on $i$, yields 
\begin{align*}
    \liminf\limits_{i \to \infty} P(y_{t_i}) 
	&\le \liminf\limits_{i \to \infty} \{P(\tilde{y}) + \sup\limits_{j \ge 1} \{\mathcal{L}_{\bar{\beta}}(x_{t_j-1}, \tilde{y}, z_{t_j-1}; \bar{M}^{t_j}) 
	\\&\;\;- P(\tilde{y})\} - \inf\limits_{j \ge 1} \{\mathcal{L}_{\bar{\beta}}(x_{t_j-1}, y_{t_j}, z_{t_j-1}; \bar{M}^{t_j}) - P(y_{t_j})\} \} 
	\\&= P(\tilde{y}) + \sup\limits_{j \ge 1} \{\mathcal{L}_{\bar{\beta}}(x_{t_j-1}, \tilde{y}, z_{t_j-1}; \bar{M}^{t_j}) - P(\tilde{y})\} 
	\\&\;\;-\inf\limits_{j\ge 1} \{\mathcal{L}_{\bar{\beta}}(x_{t_j-1}, y_{t_j}, z_{t_j-1}; \bar{M}^{t_j}) - P(y_{t_j})\}.
\end{align*}
Combining the latter with \eqref{P value liminf} then results in
\begin{align*}
    P(y^*) &\le  P(\tilde{y}) + \sup\limits_{j \ge 1} \{\mathcal{L}_{\bar{\beta}}(x_{t_j-1}, \tilde{y}, z_{t_j-1}; \bar{M}^{t_j}) - P(\tilde{y})\} \\&\;\;- \inf\limits_{j\ge1} \{\mathcal{L}_{\bar{\beta}}(x_{t_j-1}, y_{t_j}, z_{t_j-1}; \bar{M}^{t_j}) - P(y_{t_j})\}.
\end{align*}
Recalling once again that the expression $P(\tilde{y}) + \sup\limits_{j \ge 1} \{\mathcal{L}_{\bar{\beta}}(x_{t_j-1}, \tilde{y}, z_{t_j-1}; \bar{M}^{t_j}) - P(\tilde{y})\} - \inf\limits_{j\ge 1} \{\mathcal{L}_{\bar{\beta}}(x_{t_j-1}, y_{t_j}, z_{t_j-1}; \bar{M}^{t_j}) - P(y_{t_j})\} $
is a constant, it immediately follows that $P(y^*) < \infty$. 
Therefore,\\ $\mathcal{L}_{\bar{\beta}} \left(x_{t_i - 1}, y^*, z_{t_i - 1}; \bar{M}^{t_i}\right) < \infty$. 
Since $y_{t_i}$ is the minimizer of \\ $\mathcal{L}_{\bar{\beta}}(x_{t_i - 1}, y, z_{t_i - 1}; \bar{M}^{t_i})$, we have that for all $i$,
\begin{equation*}
    \mathcal{L}_{\bar{\beta}}(x_{t_i - 1}, y_{t_i}, z_{t_i - 1}; \bar{M}^{t_i}) \le \mathcal{L}_{\bar{\beta}}(x_{t_i - 1}, y^*, z_{t_i - 1}; \bar{M}^{t_i}) < \infty.
\end{equation*}
Taking the limit, we obtain that
\begin{equation*}
    \limsup\limits_{i\to\infty}\mathcal{L}_{\bar{\beta}}(x_{t_i - 1}, y_{t_i},z_{t_i - 1}; \bar{M}^{t_i}) \le \mathcal{L}_{\bar{\beta}}(x^*, y^* ,z^*; \mathbb{E}[M]) = \mathcal{L}_{\bar{\beta}}(x^*, y^* ,z^*).
\end{equation*}
Since all the elements of $\mathcal{L}_{\bar{\beta}}(x,y,z;M)$ other than $P$ are continuous, we have that $\limsup\limits_{i\to\infty} P(y_{t_i}) \le P(y^*).$
Combining this with \eqref{P value liminf} we can finally conclude that $\lim\limits_{i\to\infty} P(y_{t_i}) = P(y^*). $ 
\end{proof}

\Cref{thm:2} essentially provides us the guarantee that any accumulation point of the meta algorithm, \Cref{alg:1}, is almost surely a critical point of \eqref{eq:9}. 
All that remains is to show that the assumption  $\lim\limits_{t\to\infty} \|x_{t+1}-x_t\|^2 + \|y_{t+1}-y_t\|^2 + \|z_{t+1}-z_t\|^2 = 0$ indeed holds true under our blanket assumptions.
We do so in the proof process of our main result stated by \Cref{thm:3}.
This result requires a lengthy and highly technical lemma which we prove  in \cref{app:sec:technical}.
\begin{lem}\label{cumulative_matrix_error_effects_are_finite}
	Suppose that \Cref{assum:blanket} and \Cref{assum:2} hold true. 
	Let $\theta_t$ be chosen according to the sampling regime in \Cref{definition:2}, and let $\{x_t, y_t, z_t\}_{t \ge 0}$ be the sequence generated by \Cref{alg:1}. 
	Denote $e_{t+1} := \mathcal{L}_{\bar{\beta}}(x_t,y_t,z_t;\bar{M}^{t+1}) - \mathcal{L}_{\bar{\beta}}(x_t,y_t,z_t;\bar{M}^{t}). $
	Then, $\sum\limits_{t=1}^\infty e_{t+1} < \infty $
	almost surely.
\end{lem}

\begin{thm}\label{thm:3}
Suppose that \cref{assum:blanket}  and \Cref{assum:2} hold true. 
Let $\theta_t$ be chosen according to the sampling regime in \Cref{definition:2}, and let  $(x^*, y^*, z^*)$ be  a  cluster point of the sequence $\{(x_t, y_t, z_t)\}_{t \geq 1}$ generated by \Cref{alg:1}.
Then $x^*$ is a critical point of \eqref{eq:9} almost surely.
\end{thm}
\begin{proof}
	By \Cref{adaptive_penalty_oracle_requirements_defin}, we have the following guarantees on the sequence $\{\beta_t\}_{t\ge0}$:  There exists almost surely an index $\bar{K} : = K_{stable}$ such that: (i) $\beta$ stability: For all $t>\bar{K}$, $\beta_t = \beta_{\bar{K}}$, which for convenience we denote  by $\bar{\beta} = \beta_{\bar{K}}$; and (ii) Sufficient decrease: There exists $\rho > 0$ such that for all $t>\bar{K}$ 
	\begin{align*} 
		&g^{t+1}(x_{t+1}) - g^{t+1}(x_t) \le -\dfrac{\rho}{2}\|x_{t+1}-x_t\|^2 \text{ and } \\&\rho > \dfrac{8}{\beta_t \sigma} (\mu^2 + \nu^2),
	\end{align*}
	
	with $\mu$ and $\nu$ as in \Cref{adaptive_penalty_oracle_requirements_defin}.
	%
	%
	
	
The proof consists of 4 milestones: (1) Bounding $\|z_{t+1}-z_t\|^2$; (2) Showing that $\lim\limits_{t\to\infty} \|x_{t+1}-x_t\|^2 = 0$ implies $\lim\limits_{t\to\infty}  \left(\|y_{t+1}-y_t\|^2 + \|z_{t+1}-z_t\|^2\right) = 0$; (3) Bounding $\mathcal{L}_{\bar{\beta}}(x_{t+1},y_{t+1},z_{t+1};\bar{M}^{t+1}) - \mathcal{L}_{\bar{\beta}}(x_t,y_t,z_t;\bar{M}^{t})$ using a sequence of scalar random variables $\{d_t\}_{t \ge 0}$ satisfying that $\sum\limits_{t=0}^\infty d_{t+1} < \infty$ almost surely; (4) Showing that there almost surely exist random variables with which we can bound $\mathcal{L}_{\beta_{t+1}}(x_{t+1}, y_{t+1},z_{t+1}; \bar{M}^{t+1})$.

\noindent \textbf{Milestone 1.} 
Let $t>\bar{K}$, and note that due to \Cref{assum:blanket}, $\mathbb{E}[M]^T \mathbb{E}[M] - \sigma I \succeq 0$. Since for every matrix $Q \in \real^{n \times n}$, the eigenvalues of $Q^T Q$ are equal to the eigenvalues of $Q Q^T$, this implies $\mathbb{E}[M] \mathbb{E}[M]^T - \sigma I \succeq 0$. In particular, for any $t\geq 1$ it holds that
\begin{equation*}
    \left\langle z_{t+1} - z_t, \left(\mathbb{E}[M] \mathbb{E}[M]^T - \sigma I\right) \left(z_{t+1} - z_t\right) \right\rangle \geq 0, 
\end{equation*}
which is the same as $\sigma \|z_{t+1}-z_t\|^2 \le \|\mathbb{E}[M]^T(z_{t+1}-z_t)\|^2. $
Hence, by adding the same elements on both sides of the relation, we obtain
\begin{align}\label{eq:6}
	\sqrt{\sigma} &\|z_{t+1}-z_t\| - \|\delta_{t+1}^T z_{t+1}\| - \|\delta_t^T z_t\| \\&\le \|\mathbb{E}[M]^T(z_{t+1}-z_t)\| - \|\delta_{t+1}^T z_{t+1}\| - \|\delta_t^T z_t\|,\nonumber
\end{align}
where $\delta_t$ is defined as in \eqref{eq:error matrix}.
By the triangle inequality $\|a+b\| \ge \|a\| - \|b\|$, and the definition of $\delta_t$,
\begin{align}\label{eq:5}
	\|\mathbb{E}[M]^T(z_{t+1}-z_t)\| &- \|\delta_{t+1}^T z_{t+1}\| - \|\delta_t^T z_t\| \nonumber\\ \le& \|\mathbb{E}[M]^T(z_{t+1}-z_t) + \delta_{t+1}^T z_{t+1} - \delta_t^T z_t\| \\ =& \|(\bar{M}^{t+1})^T z_{t+1} - (\bar{M}^t)^T z_t \| \nonumber.
\end{align}
By the second part of \Cref{lem:4}, $(\bar{M}^{t+1})^T z_{t+1} = \nabla h(x_{t+1}) + \nabla \phi (x_{t+1}) - \nabla \phi (x_t) $, and consequently,
\begin{align*}
	&\|(\bar{M}^{t+1})^T z_{t+1} - (\bar{M}^t)^T z_t \|^2 \\&=  \|\nabla h(x_{t+1}) - \nabla h(x_{t}) + (\nabla \phi (x_{t+1}) - \nabla \phi (x_t)) - (\nabla \phi (x_{t}) - \nabla \phi (x_{t-1}))\|^2.
\end{align*}
Thus, using the inequality $\|a+b\|^2 \le 2\|a\|^2 + 2\|b\|^2$, we obtain that 
\begin{align}\label{eq:1}
	&\|(\bar{M}^{t+1})^T z_{t+1} - (\bar{M}^t)^T z_t \|^2  \nonumber\\&\le 2\|\nabla h(x_{t+1}) - \nabla h(x_{t}) + (\nabla \phi (x_{t+1}) - \nabla \phi (x_t))\|^2\\&+ 2\|\nabla \phi (x_{t}) - \nabla \phi (x_{t-1})\|^2. \nonumber
\end{align}

By 
\Cref{empirical_lipschitz_defin}, we have that
\begin{equation}\label{eq:3}
	\|\nabla \phi (x_{t}) - \nabla \phi (x_{t-1})\|^2 \le \nu^2\|x_{t} - x_{t-1}\|^2
\end{equation}
and 
\begin{equation}\label{eq:4}
	\|\nabla h(x_{t+1}) - \nabla h(x_{t}) + (\nabla \phi (x_{t+1}) - \nabla \phi (x_t))\|^2 \le \mu^2 \|x_{t+1} - x_{t}\|^2.
\end{equation}
Applying \eqref{eq:3} and \eqref{eq:4} to \eqref{eq:1} yields
\begin{equation}\label{eq:7}
	\|(\bar{M}^{t+1})^T z_{t+1} - (\bar{M}^t)^T z_t \|^2 \le 2 \mu^2 \|x_{t+1} - x_{t}\|^2 + 2 \nu^2 \|x_{t} - x_{t-1}\|^2.
\end{equation}
Hence, from relations \eqref{eq:6} and \eqref{eq:5} we obtain that
\begin{equation}\label{eq:23}
	\sigma\|z_{t+1}-z_t\|^2 \le \left(\|(\bar{M}^{t+1})^T z_{t+1} - (\bar{M}^t)^T z_t \| + \|\delta_{t+1}^T z_{t+1}\| + \|\delta_t^T z_t\| \right)^2.
\end{equation}
Using \eqref{eq:23}, \eqref{eq:7}, and the inequality $(a+b)^2 \le 2a^2 + 2b^2$, we  deduce the \textbf{first milestone}:
\begin{equation}	\label{x_delta_convergence_begets_all_delta_convergence}
	\|z_{t+1}-z_t\|^2  \leq \frac{4\mu^2}{\sigma}\|x_{t+1} - x_{t}\|^2 + \frac{4\nu^2}{\sigma}\|x_{t} - x_{t-1}\|^2 + \frac{2}{\sigma}(\|\delta_{t+1}^T z_{t+1}\| + \|\delta_t^T z_t\|)^2.
\end{equation}

\noindent \textbf{Milestone 2.} 
By  \cref{lem:4}, $\bar{M}^{t+1} x_{t+1} - y_{t+1} = \dfrac{1}{\bar{\beta}} (z_t - z_{t+1})$, and therefore 
\begin{equation}\label{eq:24}
	y_{t+1}-y_t = \bar{M}^{t+1} x_{t+1} - \bar{M}^{t} x_{t} + \dfrac{1}{\bar{\beta}} (z_{t+1} - z_{t}) - \dfrac{1}{\bar{\beta}} (z_{t} - z_{t-1}).
\end{equation}
Since $\bar{M}^{t+1} x_{t+1} = (\mathbb{E}[M] + \delta_{t+1})x_{t+1}$ and $\bar{M}^{t} x_{t} = (\mathbb{E}[M] + \delta_{t})x_{t}$ (from the definition of $\delta_t$), by the triangle inequality
\begin{equation}\label{eq:1060}
	\| \mathbb{E}[M](x_{t+1} - x_t)\| + \|\delta_{t+1} x_{t+1}\| + \|\delta_t x_t \| \ge \|\bar{M}^{t+1} x_{t+1} - \bar{M}^t x_t \|.
\end{equation}
Plugging \eqref{eq:1060} into \eqref{eq:24} we obtain that 
\begin{align*}
	\|y_{t+1}-y_t\| &\le \|\mathbb{E}[M]( x_{t+1} - x_{t})\| + \|\delta_{t+1}^T x_{t+1}\| + \|\delta_t ^T x_t\| \\&+  \dfrac{1}{\bar{\beta}} \|z_{t+1} - z_{t}\| + \dfrac{1}{\bar{\beta}} \|z_{t} - z_{t-1}\|.
\end{align*}
Since $\{x_t, y_t, z_t\}_{t \ge 0}$ is bounded (cf. \Cref{assum:2}) and $\delta_t \xrightarrow{t\to\infty}0$,  from the relation in \eqref{x_delta_convergence_begets_all_delta_convergence} it follows that it is sufficient to show that $\lim_{t\rightarrow\infty} \|x_{t+1}-x_t\|^2 = 0$ in order to derive that $\lim_{t\rightarrow\infty} \|x_{t+1}-x_t\|^2 + \|y_{t+1}-y_t\|^2 + \|z_{t+1}-z_t\|^2 = 0$, which concludes  the \textbf{second milestone}.

\noindent \textbf{Milestone 3.}
We will prove that there exists a sequence of scalar random variables $\{d_t\}_{t \ge 0}$ such that 
\begin{equation}\label{d_series_sum_converges}
	\sum\limits_{t=0}^\infty d_{t+1} < \infty
\end{equation}
almost surely, satisfying that
\begin{align}\label{AL_delta_bound}
	&\mathcal{L}_{\bar{\beta}}(x_{t+1},y_{t+1},z_{t+1};\bar{M}^{t+1}) - \mathcal{L}_{\bar{\beta}}(x_t,y_t,z_t;\bar{M}^{t}) \\&\le \dfrac{4}{\sigma \bar{\beta}} \left(\mu^2\|x_{t+1} - x_{t}\|^2 + \nu^2\|x_{t} - x_{t-1}\|^2 \right) - \dfrac{\rho}{2}\|x_{t+1}-x_t\|^2 + d_{t+1}. \nonumber
\end{align}
Using the telescoping sum identity,
\begin{align*}
	\mathcal{L}_{\bar{\beta}}(x_{t+1},y_{t+1},z_{t+1}; &\bar{M}^{t+1}) - \mathcal{L}_{\bar{\beta}}(x_t,y_t,z_t; \bar{M}^t) \\ 
	=& \mathcal{L}_{\bar{\beta}}(x_{t+1},y_{t+1},z_{t+1}; \bar{M}^{t+1}) - \mathcal{L}_{\bar{\beta}}(x_{t+1},y_{t+1},z_t; \bar{M}^{t+1}) \\ 
	&+ \mathcal{L}_{\bar{\beta}}(x_{t+1},y_{t+1},z_t; \bar{M}^{t+1}) - \mathcal{L}_{\bar{\beta}}(x_t,y_{t+1},z_t; \bar{M}^{t+1}) \\
	&+ \mathcal{L}_{\bar{\beta}}(x_t,y_{t+1},z_t; \bar{M}^{t+1}) - \mathcal{L}_{\bar{\beta}}(x_t,y_t,z_t; \bar{M}^{t+1}) \\
	&+ \mathcal{L}_{\bar{\beta}}(x_t,y_t,z_t; \bar{M}^{t+1}) - \mathcal{L}_{\bar{\beta}}(x_t,y_t,z_t; \bar{M}^t),
\end{align*}
we will bound each consecutive pair separately to derive a bound on the sum itself.
We begin with the bound for 
\begin{equation*}
    \mathcal{L}_{\bar{\beta}}(x_{t+1},y_{t+1},z_{t+1};\bar{M}^{t+1}) - \mathcal{L}_{\bar{\beta}}(x_{t+1},y_{t+1},z_t;\bar{M}^{t+1}).
\end{equation*}
Note that
\begin{align}\label{eq:1070}
	\mathcal{L}_{\bar{\beta}}(x_{t+1}&,y_{t+1},z_{t+1};\bar{M}^{t+1}) - \mathcal{L}_{\bar{\beta}}(x_{t+1},y_{t+1},z_t;\bar{M}^{t+1}) \nonumber \\ =&
	h(x_{t+1}) + P(y_{t+1}) - \langle z_{t+1}, \bar{M}^{t+1}x_{t+1} \rangle + \langle z_{t+1}, y_{t+1} \rangle \nonumber\\&+ \dfrac{\bar{\beta}}{2}\|\bar{M}^{t+1}x_{t+1} - y_{t+1}\|^2 \\ -& (h(x_{t+1}) + P(y_{t+1}) - \langle z_t, \bar{M}^{t+1}x_{t+1} \rangle + \langle z_t, y_{t+1} \rangle \nonumber\\&+ \dfrac{\bar{\beta}}{2}\|\bar{M}^{t+1}x_{t+1} - y_{t+1}\|^2) \nonumber \\ =&
	- \langle z_{t+1} - z_t, \bar{M}^{t+1}x_{t+1} - y_{t+1} \rangle.\nonumber
\end{align}
Plugging  $\bar{M}^{t+1}x_{t+1} - y_{t+1} = \beta_t^{-1} (z_t - z_{t+1})$ (cf. \Cref{lem:4}) 
into \eqref{eq:1070} then yields
\begin{align*}
	\mathcal{L}_{\bar{\beta}}(&x_{t+1},y_{t+1},z_{t+1};\bar{M}^{t+1}) - \mathcal{L}_{\bar{\beta}}(x_{t+1},y_{t+1},z_t;\bar{M}^{t+1})  \\&= 
	- \langle z_{t+1} - z_t, \bar{M}^{t+1}x_{t+1} - y_{t+1} \rangle \\ &= \bar{\beta}^{-1}\|z_{t+1}-z_t\|^2.
\end{align*}
Using \eqref{x_delta_convergence_begets_all_delta_convergence} we thus obtain that
\begin{align}\label{eq:44}
	\mathcal{L}_{\bar{\beta}} &(x_{t+1},y_{t+1},z_{t+1};\bar{M}^{t+1}) - \mathcal{L}_{\bar{\beta}}(x_{t+1},y_{t+1},z_t;\bar{M}^{t+1}) \\& \le \dfrac{4}{\sigma \bar{\beta}}(\mu^2\|x_{t+1} - x_{t}\|^2 + \nu^2\|x_{t} - x_{t-1}\|^2) + \dfrac{2}{\sigma \bar{\beta}}(\|\delta_{t+1}^T z_{t+1}\| + \|\delta_t^T z_t\|)^2. \nonumber
\end{align}

To evaluate $\mathcal{L}_{\bar{\beta}}(x_{t+1},y_{t+1},z_t;\bar{M}^{t+1}) - \mathcal{L}_{\bar{\beta}}(x_t,y_{t+1},z_t;\bar{M}^{t+1})$,  recall that  (cf. \eqref{g_function_defin}) $g^{t+1}(x) = h(x) - \langle z_t, \bar{M}^{t+1} x \rangle + \dfrac{\bar{\beta}}{2}\| \bar{M}^{t+1}x - y_{t+1}\|^2 + D_{\phi}(x,x_t). $
Hence, by the definition of $g^{t+1}(x)$, $\mathcal{L}_{\bar{\beta}}(x_{t+1},y_{t+1},z_t;\bar{M}^{t+1}) - \mathcal{L}_{\bar{\beta}}(x_t,y_{t+1},z_t;\bar{M}^{t+1}) = g^{t+1}(x_{t+1}) - D_\phi(x_{t+1},x_t) - g^{t+1}(x_t) +  D_\phi(x_t, x_t).$
We will bound each   component in this equation separately.

First, by \cref{adaptive_penalty_oracle_requirements_defin}, for every $t > \bar{K}$ it holds that $	g^{t+1}(x_{t+1}) - g^{t+1}(x_t) \le -\dfrac{\rho}{2}\|x_{t+1}-x_t\|^2. $
Then, from the definition of $D_{\phi}(\cdot,\cdot)$ we have that $D_\phi(x,x_t) = \phi(x) - \phi(x_t) - \langle \nabla \phi(x_t), x-x_t\rangle, $
which clearly implies that $D_\phi(x_t,x_t) = 0 $.
At last, since $\phi$ is convex, we can apply the gradient inequality to $\phi$ and derive that $	D_\phi(x_{t+1}, x_t) \ge 0. $
Using these facts, it follows that
\begin{align}\label{eq:46}
	\mathcal{L}_{\bar{\beta}}(x_{t+1},y_{t+1},z_t;\bar{M}^{t+1}) &- \mathcal{L}_{\bar{\beta}}(x_t,y_{t+1},z_t;\bar{M}^{t+1}) \\ 
	=& g^{t+1}(x_{t+1}) - g^{t+1}(x_t) - D_\phi(x_{t+1},x_t) + D_{\phi}(x_t, x_t) \nonumber \\ 
	=& \nonumber g^{t+1}(x_{t+1}) - g^{t+1}(x_t) - D_\phi(x_{t+1},x_t) \\ 
	\le& -\dfrac{\rho}{2}\|x_{t+1}-x_t\|^2. \nonumber
\end{align}

To bound $\mathcal{L}_{\bar{\beta}}(x_t,y_{t+1},z_t;\bar{M}^{t+1}) - \mathcal{L}_{\bar{\beta}}(x_t,y_t,z_t;\bar{M}^{t+1})$, note that $y_{t+1}$ is the minimizer of $\mathcal{L}_{\bar{\beta}}(x_t,y,z_t;\bar{M}^{t+1})$ (we remind the reader that for $t\ge \bar{K}$, $\beta_t = \bar{\beta}$). Hence,
\begin{gather}\label{eq:43}
	\mathcal{L}_{\bar{\beta}}(x_t,y_{t+1},z_t;\bar{M}^{t+1}) - \mathcal{L}_{\bar{\beta}}(x_t,y_t,z_t;\bar{M}^{t+1}) \le 0.
\end{gather}
Finally, to bound  
\begin{equation*}
    e_{t+1} := \mathcal{L}_{\bar{\beta}}(x_t,y_t,z_t;\bar{M}^{t+1}) - \mathcal{L}_{\bar{\beta}}(x_t,y_t,z_t;\bar{M}^{t})
\end{equation*}
we use \cref{cumulative_matrix_error_effects_are_finite}. 


Combining \eqref{eq:44}, \eqref{eq:46}, and \eqref{eq:43}, we have that
\begin{align*}
	\mathcal{L}_{\bar{\beta}}&(x_{t+1},y_{t+1},z_{t+1};\bar{M}^{t+1}) - \mathcal{L}_{\bar{\beta}}(x_t,y_t,z_t;\bar{M}^{t}) \\
	\le& \dfrac{4}{\sigma \bar{\beta}} \left(\mu^2\|x_{t+1} - x_{t}\|^2 + \nu^2\|x_{t} - x_{t-1}\|^2 \right) + \dfrac{2}{\sigma \bar{\beta}}\left(\|\delta_{t+1}^T z_{t+1}\| + \|\delta_t^T z_t\| \right)^2 \\
	&- \dfrac{\rho}{2}\|x_{t+1}-x_t\|^2 + e_{t+1}. \nonumber
\end{align*}
Set $d_{t+1}$ from \eqref{d_series_sum_converges} to be $d_{t+1} = \dfrac{2}{\sigma \bar{\beta}} \left(\|\delta_{t+1}^T z_{t+1}\| +\|\delta_t^T z_t\| \right)^2 + e_{t+1}, $
so that
\begin{align}\label{eq:47}
	&\mathcal{L}_{\bar{\beta}}(x_{t+1},y_{t+1},z_{t+1};\bar{M}^{t+1}) - \mathcal{L}_{\bar{\beta}}(x_t,y_t,z_t;\bar{M}^{t}) \\& \le \dfrac{4}{\sigma \bar{\beta}} \left(\mu^2\|x_{t+1} - x_{t}\|^2 + \nu^2\|x_{t} - x_{t-1}\|^2 \right) - \dfrac{\rho}{2}\|x_{t+1}-x_t\|^2 + d_{t+1}. \nonumber
\end{align}
The relation in \eqref{eq:47} is exactly the required bound \eqref{AL_delta_bound}.

To prove that the sequence $\{d_{t+1}\}_{t\geq 1}$ indeed satisfies \eqref{d_series_sum_converges}, note that by \Cref{assum:2} the sequence $\{z_t\}_{t\ge0}$ is bounded -- that is, there exists a constant $D>0$, such  that $\|z_t\| \le D$ for all $t$. 
Using this and the induced norm inequality, we obtain that $d_{t+1} = \dfrac{2}{\sigma \bar{\beta}} \left(\|\delta_{t+1}^T z_{t+1}\| +\|\delta_t^T z_t\| \right)^2 + e_{t+1} \le \dfrac{2D^2}{\sigma \bar{\beta}} \left(\|\delta_{t+1}\| +\|\delta_t\| \right)^2 + e_{t+1}. $
Applying the inequality $(a+b)^2 \le 2a^2 + 2b^2$,
\begin{equation*}
	d_{t+1} \le \dfrac{2D^2}{\sigma \bar{\beta}} \left(\|\delta_{t+1}\| +\|\delta_t\| \right)^2 + e_{t+1} \le \dfrac{4D^2}{\sigma \bar{\beta}} \left(\|\delta_{t+1}\|^2 +\|\delta_t\|^2 \right) + e_{t+1}.
\end{equation*}

As established above,  our choice of $\theta_t$ implies 
\begin{equation*}
    \mathbb{P} \left(\|\delta_t\| > \alpha t^{-(0.5 + 0.25\epsilon)} \; i.o. \right) = 0
\end{equation*}
for some $\alpha > 0$.
Combined with \cref{cumulative_matrix_error_effects_are_finite}, $\sum_{t=\bar{K}+1}^{\infty}d_{t+1} < \infty$ holds almost surely.
This concludes the proof of \eqref{d_series_sum_converges}, and the third milestone of the proof.

\noindent \textbf{Milestone 4.} 
Utilizing \eqref{eq:47}, for any $q \ge p+1 \ge \bar{K}$ we have that
\begin{align*}
	&\mathcal{L}_{\bar{\beta}}(x_q,y_q,z_q;\bar{M}^q) - \mathcal{L}_{\bar{\beta}}(x_p,y_p,z_p;\bar{M}^{p}) \\& \le \dfrac{1}{2}\sum\limits_{t=p}^{q} \left(\dfrac{8\mu^2}{\sigma \bar{\beta}} - \rho \right)\|x_{t+1}-x_t\|^2 + \dfrac{1}{2}\sum\limits_{t=p}^{q} \dfrac{8 \nu^2}{\sigma \bar{\beta}}\|x_t-x_{t-1}\|^2 +\sum\limits_{t=p}^{q} d_{t+1}.
\end{align*}
Since $\sum\limits_{t=p}^{q} \|x_t-x_{t-1}\|^2  = \sum\limits_{t=p-1}^{q-1} \|x_{t+1}-x_{t}\|^2 ,$
it follows that
\begin{align*}
	&\mathcal{L}_{\bar{\beta}}(x_q,y_q,z_q;\bar{M}^q) - \mathcal{L}_{\bar{\beta}}(x_p,y_p,z_p;\bar{M}^{p}) \\& \le  \dfrac{1}{2}\sum\limits_{t=p}^{q}\left(\dfrac{8\mu^2}{\sigma \bar{\beta}} - \rho \right)\|x_{t+1}-x_t\|^2 + \dfrac{1}{2}\sum\limits_{t=p-1}^{q-1} \dfrac{8 \nu^2}{\sigma \bar{\beta}}\|x_{t+1}-x_{t}\|^2 +\sum\limits_{t=p}^{q} d_{t+1}.
\end{align*}
Separating the last element from the first summation, and the first element from the second summation, yields
\begin{align*}
	&\mathcal{L}_{\bar{\beta}}(x_q,y_q,z_q;\bar{M}^q) - \mathcal{L}_{\bar{\beta}}(x_p,y_p,z_p;\bar{M}^{p}) \\& \le  \dfrac{1}{2}\sum\limits_{t=p}^{q-1} \left(\dfrac{8\mu^2 + 8\nu^2}{\sigma \bar{\beta}} - \rho \right)\|x_{t+1}-x_t\|^2 + \dfrac{1}{2} \left(\dfrac{8\mu^2}{\sigma \bar{\beta}} - \rho\right)\|x_{q+1}-x_q\|^2 \\ &\quad + \dfrac{4\nu^2}{\sigma \bar{\beta}}\|x_{p}-x_{p-1}\|^2 +\sum\limits_{t=p}^{q} d_{t+1}.
\end{align*}
By \cref{adaptive_penalty_oracle_requirements_defin},
$ 8(\sigma \bar{\beta})^{-1} \mu^2 - \rho \leq 8(\sigma \bar{\beta})^{-1} (\mu^2 + \nu^2) -\rho < 0$ almost surely.
Therefore, denoting 
\begin{equation*}
C_1 := 4 (\sigma \bar{\beta})^{-1} (\mu^2 + \nu^2) - \tfrac{\rho}{2},
\end{equation*}
which satisfies $C_1 < 0$ almost surely,
we have that
%
\begin{equation}\label{eq:1030}
	\begin{aligned}
		\mathcal{L}_{\bar{\beta}}&(x_q,y_q,z_q;\bar{M}^q) - \mathcal{L}_{\bar{\beta}}(x_p,y_p,z_p;\bar{M}^p) \\
		&\leq \sum\limits_{t=p}^{q-1} C_1 \|x_{t+1}-x_t\|^2 + \dfrac{4\nu^2}{\sigma \bar{\beta}}\|x_{p}-x_{p-1}\|^2
		+\sum\limits_{t=p}^{q} d_{t+1}.
	\end{aligned}
\end{equation}
Rearranging elements in \eqref{eq:1030}  with $q=t$ and $p=\bar{K}+1$ and using the scalar random variable $C_2 = \frac{4\nu^2}{\sigma \bar{\beta}}\|x_{\bar{K}+1} -x_{\bar{K}}\|^2 + \mathcal{L}_{\bar{\beta}}(x_{\bar{K}+1},y_{\bar{K}+1},z_{\bar{K}+1}; \bar{M}^{\bar{K}+1}), $
we have that $\mathcal{L}_{\bar{\beta}}(x_{t},y_{t},z_{t}; \bar{M^{t}})  \le \sum\limits_{i=\bar{K} + 1}^{t - 1} C_1 \|x_{i+1}-x_i\|^2 + \sum\limits_{i=\bar{K}+1}^{t}d_{i+1} + C_2. $
By \Cref{assum:2}, there exists $D > 0$ such that $\|x_t\|$ is bounded by $D$ for all $t$. Therefore, 
\begin{equation}\label{eq:1040}
	\dfrac{4\nu^2}{\sigma \bar{\beta}}\|x_{\bar{K}+1} -x_{\bar{K}}\|^2 \le \dfrac{16\nu^2 D^2}{\sigma \bar{\beta}} < \infty.
\end{equation}
Since $P$ is proper, there exists $\tilde{y}$ such that $P(\tilde{y}) < \infty$.
Since $P(\tilde{y}) < \infty$ and all other components are continuous, $\mathcal{L}_{\bar{\beta}}(x_{\bar{K}}, \tilde{y},z_{\bar{K}};\bar{M}^{\bar{K}+1}) < \infty. $
Note that $y_{\bar{K}+1}$ is the minimizer of $\mathcal{L}_{\bar{\beta}}(x_{\bar{K}}, y, z_{\bar{K}}; \bar{M}^{\bar{K}+1})$ with respect to $y$, and hence, $\mathcal{L}_{\bar{\beta}}(x_{\bar{K}}, y_{\bar{K}+1}, z_{\bar{K}}; \bar{M}^{\bar{K}+1}) \le  \mathcal{L}_{\bar{\beta}}(x_{\bar{K}}, \tilde{y},z_{\bar{K}};\bar{M}^{\bar{K}+1}) < \infty,$
and therefore, $P(y_{\bar{K}+1}) < \infty$. 
Since $P$ is the only component of $\mathcal{L}_{\bar{\beta}}$ that can take the value $+\infty$ for finite values of $x,y,z$, and $P(y_{\bar{K}+1}) < \infty$, it follows that $\mathcal{L}_{\bar{\beta}}(x_{\bar{K}+1}, y_{\bar{K}+1}, z_{\bar{K}+1}; \bar{M}^{\bar{K}+1}) < \infty.$
Combining the latter with \eqref{eq:1040}  implies that with probability 1, $ C_2 < \infty.$
This concludes that there exist random variables $C_1, C_2$, such that $C_1 < 0$, $C_2 < \infty$ almost surely, and 
\begin{equation}\label{AL_bound}
	\mathcal{L}_{\bar{\beta}}(x_{t},y_{t},z_{t}; \bar{M}^{t}) \le \sum\limits_{i=\bar{K}+1}^{t-1} C_1 \|x_{i+1}-x_i\|^2 + \sum\limits_{i=\bar{K}+1}^{t} d_{i+1} + C_2.
\end{equation}

\noindent \textbf{Main proof using the four milestones.} 
Let $\{(x_{t_i}, y_{t_i}, z_{t_i})\}_{i\ge0}$ be a subsequence converging to  $(x^*, y^*, z^*)$. 
Note the following facts: (i) $\mathcal{L}_{\bar{\beta}}(\cdot, \cdot, \cdot; A)$ is continuous in its last argument; (ii) $\bar{M}^{t}\xrightarrow{t\to\infty} \mathbb{E}[M]$; (iii) $\mathcal{L}_{\bar{\beta}}$ is lower semicontinuous; (iv) $P$ is proper, and therefore $\mathcal{L}_{\bar{\beta}}$ is proper.
%
%
%
Using these facts, and the notation $\mathcal{L}_{\bar{\beta}}(\cdot,\cdot,\cdot,\mathbb{E}[M])=\mathcal{L}_{\bar{\beta}}$, the following inequalities hold true. 
\begin{equation}\label{eq:1020}
	\liminf_{i\to\infty} \mathcal{L}_{\bar{\beta}}(x_{t_i}, y_{t_i}, z_{t_i}; \bar{M^{t_i}}) \ge \mathcal{L}_{\bar{\beta}}(x^*, y^*, z^*) > -\infty.
\end{equation}
Combining \eqref{eq:1020} with \eqref{AL_bound}, we deduce that there exists a sequence of random variables $\{d_t\}_{t \ge 0}$ such that $\sum\limits_{i=\bar{K}+1}^{\infty} C_1  \|x_{i+1}-x_i\|^2 + \sum\limits_{i=\bar{K}+1}^{\infty} d_{i+1} + C_2 > -\infty, $
which implies,
by \eqref{d_series_sum_converges} 
that $\sum\limits_{i=\bar{K}+1}^{\infty} C_1  \|x_{i+1}-x_i\|^2 + C_2 > -\infty $ almost surely.
Since $C_1 < 0$ and $C_2 < \infty$, it follows that $\lim\limits_{t\to\infty} \|x_{t+1}-x_t\|^2 = 0 $ almost surely.
Finally, by \eqref{x_delta_convergence_begets_all_delta_convergence}, it follows that  $\lim\limits_{t\to\infty} \|x_{t+1}-x_t\|^2 + \|y_{t+1}-y_t\|^2 + \|z_{t+1}-z_t\|^2 = 0, $
which readily implies the required by invoking \Cref{thm:2}.
\end{proof}

\begin{remark}[On convergence of the whole sequence]\label{rem:whole-sequence}
\Cref{thm:3} gives subsequential convergence to a critical point.
Under additional structure, such as the Kurdyka--{\L}ojasiewicz property,
this is classically upgraded to convergence of the whole sequence by a
finite-length argument, which requires a sufficient decrease property and
a relative-error condition for a fixed Lyapunov function; see, e.g.,
\cite{BST13,bolte2018nonconvex}.
In our setting, \eqref{eq:47} provides only an approximate descent
of the moving augmented Lagrangian
$\mathcal{L}_{\bar{\beta}}(\cdot,\cdot,\cdot;\bar{M}^{t})$, up to the summable errors
$d_{t+1}$. Moreover, for a Lyapunov function built from the fixed augmented Lagrangian
$\mathcal{L}_{\bar{\beta}}(\cdot,\cdot,\cdot;\mathbb{E}[M])$, the relative-error condition can
only be established up to additive terms of order $\|\delta_t\|$, for which
our sampling regime guarantees square-summability only. Consequently, the
classical finite-length argument does not apply directly.
Extending the classical argument to accommodate
the estimation errors is left for future research.
\end{remark}

\subsection{ISAD for Problems with Bounded Hessian} \label{sec: A refined method to bounded hessian problems}

The meta \Cref{alg:1} contains two components that can be chosen by the implementer -- an APO with the requirements in \cref{adaptive_penalty_oracle_requirements_defin} and a convex function $\phi$.
In this section we briefly present a simple implementation of an APO and a function $\phi$ for  objective functions with a bounded Hessian, i.e., for when there exists $\gamma>0$ such that $-\gamma I \preceq \nabla^2 h(x) \preceq \gamma I$ for any $x \in \real^n$. 
The details of the implementation, including formal proofs, appear in \cite{greenstein2023augmented}.

Indeed, functions with a bounded Hessian are used to model a  variety of problems, especially due to the equivalence between bounded Hessian and Lipschitz continuity of the gradient, see, e.g., \cite[Theorem 4.20]{B14}.

Using the bound on the Hessian, the function $\phi$ is chosen to be $\phi(x) = \frac{\gamma}{2}\|x\|^2 - h(x), $
which leads to a strong convexity property of $g^{t}$, and subsequently, to a closed-form solution to the $x$ update.
It additionally facilitates  fast converging gradient-based methods to calculate the $x$ step (see \cite[Theorem 3.10]{bubeck2015convex}).
Hence, the $x$ update in \Cref{alg:ALAS} invoked by the meta-algorithm can be transformed to
\begin{equation*}
\begin{aligned}
	x_{t+1}\in \argmin\limits_x \bigl\{&- \langle z_t, \bar{M}^{t+1} x \rangle + \dfrac{\beta_t}{2}\| \bar{M}^{t+1}x - y_{t+1}\|^2 \\
	&+ \dfrac{\gamma}{2} \|x-x_t\|^2 + \langle \nabla h(x_t), x - x_t \rangle\bigr\}.
\end{aligned}
\end{equation*}
Therefore, a simple APO implementation satisfying the required properties is:
Given $\beta\in\real$, $M \in\real^{n \times n}$, $\gamma>0$, $\epsilon > 0$ and $\tilde{\sigma} = \lambda_{\min}(M^T M)$, 
if one of the following holds true: (i) $\tilde{\sigma} = 0$; (ii) $(1 + 0.5\epsilon) \cdot 40 \gamma^2 (\tilde{\sigma}\beta)^{-1} < \tilde{\sigma} \beta + \gamma < (1 + 2\epsilon) \cdot 40 \gamma^2 (\tilde{\sigma}\beta)^{-1}$,
then return $\beta$. Otherwise, update $\beta = (-\gamma + \sqrt{\gamma^2 + (1 + \epsilon)160 \gamma^2})(2 \tilde{\sigma})^{-1}$.

\section{Numerical Experiments}\label{sec:experiments}

We evaluate the performance of the ISAD algorithm on a multi-class classification problem with group-sparse constraints, examining both the i.i.d.\ sampling regime and convergent matrix sequences.

\subsection{Experimental Setup}

\textbf{Problem formulation.} We solve instances of \eqref{eq:9} where $h(x)$ is the multi-class cross-entropy loss and $P$ encodes group-sparsity constraints. Specifically, for the Covertype dataset \cite{covertype_31} with $N=581012$ samples, $d=54$ features, and $C=7$ classes, we parametrize $x \in \mathbb{R}^{dC}$ as a flattened weight matrix $W \in \mathbb{R}^{d \times C}$, where each column corresponds to one class. The random matrix $M \in \mathbb{R}^{d \times d}$ represents a PCA embedding, with $\mathbb{E}[M]$ applied to each class independently via the block-diagonal structure $I_C \otimes \mathbb{E}[M]$.

The constraint function $P$ is the indicator function for group sparsity: for each class $c$, we require that among the top $m_{\text{active}} = 5$ components of $(\mathbb{E}[M]x)^{(c)}$, at most $k=2$ are nonzero, while the remaining $(d - m_{\text{active}})$ components are unconstrained. This translates to a disjoint group-sparsity structure on $\mathbb{E}[M]x$, encouraging the classifier to utilize only a sparse subset of principal components per class.

\textbf{Matrix sampling.} We consider two distinct sampling regimes:

\emph{I.I.D.\ sampling.} We first construct a reference PCA basis by subsampling 5000 data points and computing their principal components, yielding a $d \times d$ matrix. During algorithm execution, at each iteration we sample a $d \times d$ matrix $M^i_{\text{base}}$ by: (i) subsampling 100 data points; (ii) computing PCA on this subsample; (iii) aligning the resulting components to the reference basis to maintain consistent row ordering across samples. This alignment is crucial: since PCA orders components by eigenvalue magnitude, different subsamples may yield different orderings, which would distort the empirical average $\bar{M}^t$ away from approximating the true PCA embedding of the full dataset. The random matrix $M^i \in \mathbb{R}^{n \times n}$ (where $n = dC$) is then constructed via the block-diagonal structure $M^i = I_C \otimes M^i_{\text{base}}$, which applies the same PCA transformation independently to each of the $C$ classes. The matrices $\{M^i\}$ are i.i.d., as they are independently drawn from the same sampling distribution (conditioned on the fixed reference). We sample 2 matrices in the first iteration and 1 matrix in each subsequent iteration, yielding $\theta_t = t + 1$ total samples after iteration $t$, consistent with the regime $\theta_t = t^{1+\epsilon}$ for $\epsilon \approx 0$. The estimator $\bar{M}^t$ is formed via running average as in Line 3 of Algorithm 2. We write $\widehat{\mathbb{E}}_{201}[M]$ for the average of the 201 matrices sampled in this way. The estimator $\bar{M}^t$ is the average of the same matrices seen up to iteration $t$, and the reported target gap is the gap to $\widehat{\mathbb{E}}_{201}[M]$.

\emph{Convergent sequences.} We employ incremental PCA \cite{ross2008incremental}, processing the dataset sequentially in batches. At iteration $t$, a base matrix $M^t_{\text{base}} \in \mathbb{R}^{d \times d}$ is obtained by applying incremental PCA to the first $t \cdot b$ data points, where $b=100$ is the batch size, and the full matrix $M^t \in \mathbb{R}^{n \times n}$ is constructed as $M^t = I_C \otimes M^t_{\text{base}}$. Here, no averaging is used; we set $\bar{M}^{t+1} = M^{t+1}$ directly. The ``ground truth'' is taken as $M^{200}$, the PCA matrix obtained after processing $200 \cdot b =$ 20,000 rows---a substantial but incomplete portion of the dataset.

While our convergence analysis (Sections \ref{sec: Properties of the Sampling Mechanism}--\ref{sec: Meta Algorithm Convergence}) is developed for the i.i.d.\ setting, the core technical requirements are that the estimation error $\delta_t = \bar{M}^t - \mathbb{E}[M]$ converges to zero and satisfies $\sum_{t=1}^\infty \|\delta_t\|^2 < \infty$ almost surely (cf.\ Lemmas \ref{lem:not_io_then_bounded}--\ref{lem:2}). Convergent sequences satisfying these properties may therefore admit similar theoretical guarantees, though a formal extension is beyond the scope of this work. From a practical perspective, such sequences arise naturally in streaming data applications where each $M^t$ refines the previous estimate as more data arrives. This experiment evaluates ISAD's empirical performance in this important setting.

\textbf{Algorithmic parameters.} We implement the Meta-ISAD algorithm (\Cref{alg:1}) using the alternating update (\Cref{alg:ALAS}) and the general APO (\Cref{general_beta_oracle}). The initial penalty is $\beta_0 = 0.0625$, and the APO parameter is set to $\epsilon = 10^{-6}$. For the Bregman function, we use $\phi(x) = \frac{1}{2}\|x\|^2$, yielding the Euclidean Bregman divergence $D_\phi(x,x_0) = \frac{1}{2}\|x - x_0\|^2$.

The $x$-update in \Cref{alg:ALAS} (Line 2) is solved to high precision using the Adam optimizer with learning rate $0.001$ and standard parameters $(\beta_1, \beta_2) = (0.9, 0.999)$. We terminate the inner optimization when $\|\nabla g^{t+1}(x)\| < 10^{-4}$ or after 50,000 iterations, where $g^{t+1}$ is defined in \eqref{g_function_defin}. This ensures that errors from the inner optimization do not accumulate in the outer AL iterations.

\textbf{Experimental design.} We conduct four experiments organized into two high-level suites:
\begin{enumerate}
    \item \emph{I.I.D.\ suite:} (Exp.\ 1) ISAD with i.i.d.\ sampled matrices and running average; (Exp.\ 2) ISAD with constant matrix $\widehat{\mathbb{E}}_{201}[M]$.
    \item \emph{Convergent suite:} (Exp.\ 3) ISAD with incremental PCA sequence (no averaging); (Exp.\ 4) ISAD with constant matrix $M^{200}$ (convergence target).
\end{enumerate}
The comparison between Experiments 1 and 2 evaluates whether stochastic averaging recovers deterministic performance under i.i.d.\ sampling, while Experiments 3 and 4 assess convergence behavior under a non-i.i.d.\ but convergent sequence. All experiments use $T=200$ meta-iterations and identical initialization ($x_0 = z_0 = 0$).

\textbf{Convergence metrics.} We track three quantities at each iteration: (i) the \emph{target feasibility gap}, measuring constraint violation with respect to the experimental ground truth---for the i.i.d.\ suite this is $\|\widehat{\mathbb{E}}_{201}[M]x_t - y_t\|^2$, while for the convergent suite this is $\|M^{200}x_t - y_t\|^2$ (using the final incremental PCA matrix); (ii) the \emph{estimator feasibility gap} $\|\bar{M}^t x_t - y_t\|^2$, measuring constraint violation with respect to the current matrix estimate (tracked only in the stochastic experiments); (iii) the classification loss $h(x_t)$, measuring objective performance. The estimator feasibility gap indicates how well the algorithm performs given the available data: a small value demonstrates effective constraint satisfaction relative to $\bar{M}^t$, while a large discrepancy between $\|\bar{M}^t x_t - y_t\|^2$ and the target feasibility gap indicates that the primary source of constraint violation is the estimation error $\delta_t = \bar{M}^t - M^{\text{target}}$ (where $M^{\text{target}}$ is $\widehat{\mathbb{E}}_{201}[M]$ or $M^{200}$ depending on the suite) rather than the algorithm's iterates. These metrics provide insight into both the algorithm's convergence behavior and the quality of the matrix estimation process.

\subsection{Experimental Results}

We present the convergence behavior of ISAD across the four experimental settings described in Section~6.1, examining both feasibility gap metrics and classification loss over 200 iterations.

\textbf{I.I.D.\ sampling suite.} Figure~\ref{fig:iid_results} displays the results for Experiments 1 and 2. Both experiments exhibit rapid convergence in all tracked metrics. The deterministic baseline (Exp.\ 2, constant $\widehat{\mathbb{E}}_{201}[M]$) achieves near-machine-precision feasibility ($\approx 10^{-12}$) within 15 iterations and reaches a final classification loss of $h(x_{198}) = 1.168$. The stochastic variant (Exp.\ 1, i.i.d.\ sampling with averaging) demonstrates comparable performance: the estimator feasibility gap $\|\bar{M}^t x_t - y_t\|^2$ reaches $\approx 10^{-10}$ by iteration 50, while the target feasibility gap $\|\widehat{\mathbb{E}}_{201}[M]x_t - y_t\|^2$ is $2.4 \times 10^{-5}$ at iteration 198, and the final classification loss is $h(x_{198}) = 1.187$. 

The small discrepancy between $\|\bar{M}^t x_t - y_t\|^2$ and $\|\widehat{\mathbb{E}}_{201}[M]x_t - y_t\|^2$ in Experiment 1 demonstrates the value of tracking both metrics: the algorithm effectively satisfies the constraints relative to its current matrix estimate $\bar{M}^t$, with the remaining constraint violation attributable primarily to estimation error $\delta_t = \bar{M}^t - \widehat{\mathbb{E}}_{201}[M]$ rather than algorithmic performance. The near-identical final loss values (1.187 vs.\ 1.168) confirm that stochastic averaging successfully recovers deterministic performance under i.i.d.\ sampling.

\begin{figure}[ht]
\centering
\includegraphics[width=0.8\textwidth]{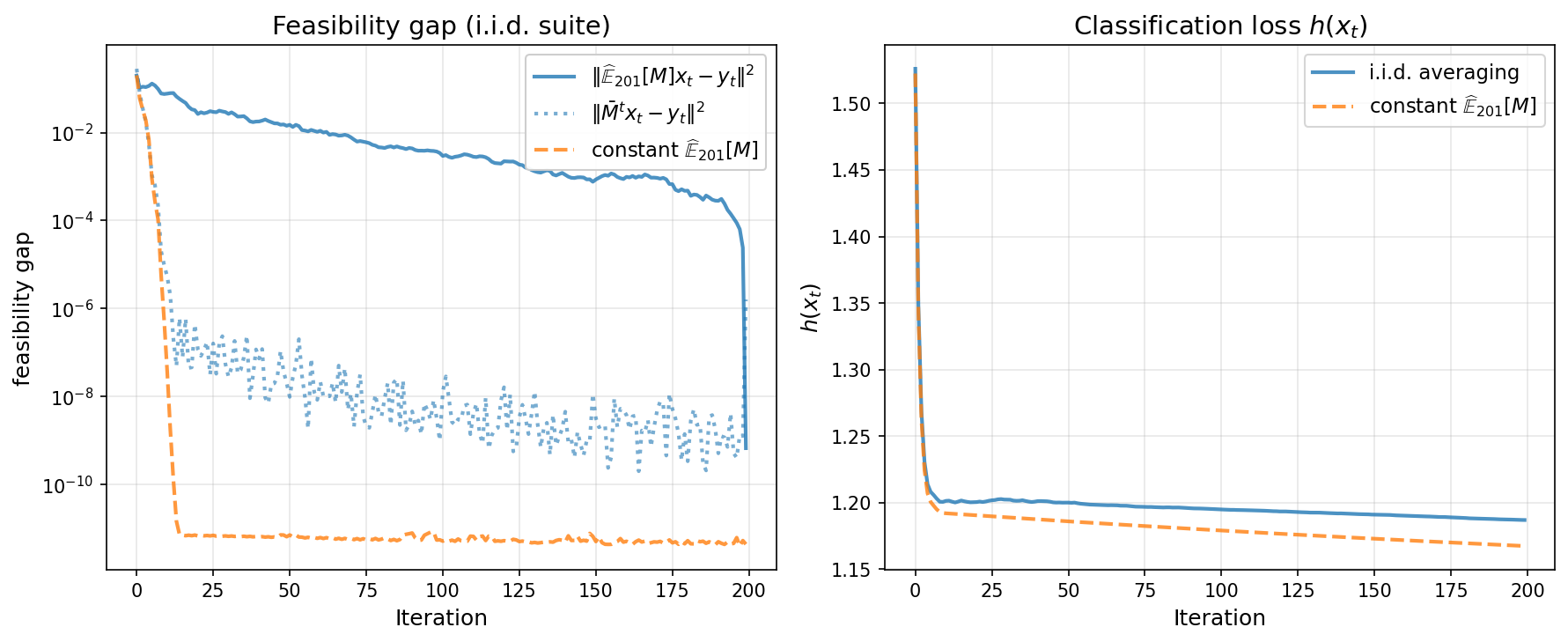}
\caption{I.i.d.\ suite across 200 iterations: $\|\widehat{\mathbb{E}}_{201}[M]x_t-y_t\|^2$, $\|\bar{M}^t x_t-y_t\|^2$, and constant $\widehat{\mathbb{E}}_{201}[M]$, together with $h(x_t)$.}
\label{fig:iid_results}
\end{figure}

\textbf{Convergent sequence suite.} Figure~\ref{fig:convergent_results} presents results for Experiments 3 and 4. The deterministic baseline (Exp.\ 4, constant $M^{200}$) achieves a final classification loss of $h(x_{198}) = 1.129$ with near-zero feasibility gap ($\approx 10^{-10}$). The convergent sequence experiment (Exp.\ 3, incremental PCA without averaging) exhibits markedly different behavior in its feasibility metrics while maintaining competitive classification performance.

Most notably, in Experiment 3 the estimator feasibility gap $\|\bar{M}^t x_t - y_t\|^2$ decreases smoothly from 0.316 to $4.1 \times 10^{-6}$ by iteration 198, demonstrating effective constraint satisfaction relative to the current matrix $\bar{M}^t = M^t$. In contrast, the target feasibility gap $\|M^{200}x_t - y_t\|^2$ oscillates between 1.0 and 4.9 throughout most of the execution (iterations 5--150), only dropping to 0.034 in the final iterations as the incremental PCA sequence converges to its limit. This stark divergence between the two feasibility metrics provides clear empirical evidence that the algorithm performs well given its available data---the large target feasibility gap reflects the non-stationarity of the matrix sequence rather than poor algorithmic performance.

Despite operating in this challenging non-stationary regime, Experiment 3 achieves a final classification loss of $h(x_{198}) = 1.155$, remaining competitive with the deterministic baseline (1.129) and occasionally outperforming it at intermediate iterations. This demonstrates that ISAD exhibits robust convergence behavior even under convergent (non-i.i.d.) matrix sequences, supporting the practical applicability of the algorithm beyond the i.i.d.\ setting formally analyzed in our theoretical framework.

\begin{figure}[ht]
\centering
\includegraphics[width=0.8\textwidth]{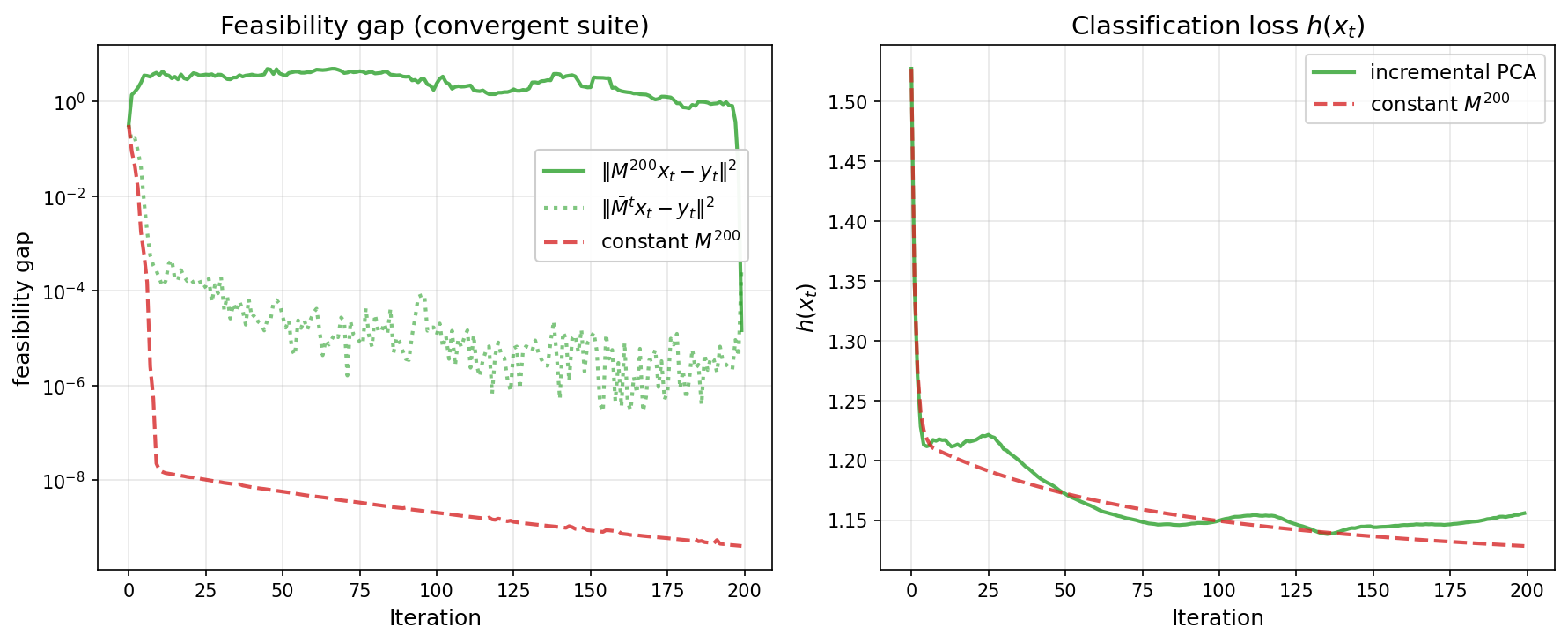}
\caption{Convergent suite across 200 iterations: $\|M^{200}x_t-y_t\|^2$, $\|\bar{M}^t x_t-y_t\|^2$, and constant $M^{200}$, together with $h(x_t)$.}
\label{fig:convergent_results}
\end{figure}

\textbf{Penalty adaptation.} Across all four experiments, the adaptive penalty parameter $\beta_t$ stabilizes within the first 20--30 iterations, consistent with the almost-sure finite stabilization guaranteed by \Cref{adaptive_penalty_oracle_requirements_defin}.

\section{Conclusion}\label{sec:conclusion}
This paper presents a meta-algorithm for minimizing nonconvex, nonsmooth stochastic composite problems. 
Our approach motivates future research on nonconvex, nonsmooth AL methods with advanced sampling strategies for non-i.i.d. scenarios or specific matrix distributions, as well as regret-based analysis for online or streaming settings. 
Another interesting, yet more challenging, direction is the study of nonlinear random compositions analogous to recent advances in nonconvex AL methodology.
A further direction is to establish verifiable sufficient conditions for the boundedness of the iterates (\Cref{assum:2}) in the stochastic setting, for instance by adapting the coercivity-based argument of~\cite{LP15} to the approximate-descent structure of our analysis (cf.\ \Cref{rem:boundedness}).
Finally, one may establish convergence of the whole sequence by extending the classical Kurdyka--{\L}ojasiewicz argument to accommodate the estimation errors (cf.\ \Cref{rem:whole-sequence}).

%
%
%

\appendix





\section{Mathematical Ingredients Proofs}
\begin{proof}[Proof of {\cref{lemma_empirical_lipschitz_bounded_by_lipschitz}}]
If for every $t$, $x_{t+1} = x_t$, then $L^e_F (\{x_t\}_{t \ge 0}) = 0$. Since $L_F \ge 0$, the result trivially holds true.
Otherwise, there exists $t$ such that $x_{t+1} \ne x_t$. Since for every $x_{t+1} \ne x_t$, 
\begin{equation*}
    \dfrac{\|F(x_{t+1}) - F(x_t)\|}{\|x_{t+1} - x_t\|} \le \sup\limits_{y \ne x} \dfrac{\|F(x) - F(y) \|}{\|x - y\|},
\end{equation*}
it follows that
\begin{equation*}
    L^e_F(\{x_t\}_{t \ge 0}) = \sup\limits_{t \ge 0, x_{t+1} \ne x_t} \dfrac{\|F(x_{t+1}) - F(x_t)\|}{\|x_{t+1} - x_t\|} \le \sup\limits_{y \ne x} \dfrac{\|F(x) - F(y) \|}{\|x - y\|} = L_F. 
\end{equation*}
\end{proof}

\begin{proof}[Proof of {\cref{lemma_empirical_lipschitz_running_estimate}}]
Let
\begin{equation*}
S := \left\{ s_t := \frac{\|F(x_{t+1})-F(x_t)\|}{\|x_{t+1}-x_t\|} \;:\; t\ge0,\; x_{t+1}\neq x_t \right\}.
\end{equation*}
If $S=\emptyset$ then $x_{t+1}=x_t$ for all $t$, hence by definition $l_t\equiv 0$ and $L^e_F(\{x_t\}_{t\ge0})=0$, so the claim holds. Thus assume $S\neq\emptyset$.

By construction $l_{t+1}=\max\{l_t,s_t\}$ whenever $x_{t+1}\neq x_t$, and $l_{t+1}=l_t$ otherwise. Therefore $\{l_t\}$ is nondecreasing.

Set 
\begin{equation*}
s^* := \sup S.
\end{equation*}
We treat two cases.

\medskip\noindent\textbf{Case 1: $s^*<\infty$.}  
By construction $l_{t+1}=\max\{l_t,s_t\}$ whenever $x_{t+1}\neq x_t$, and $l_{t+1}=l_t$ otherwise.  
An induction argument shows that
\begin{equation*}
l_t = \max\{l_0, s_0,s_1,\ldots,s_{t-1}\}, \qquad t\ge 1,
\end{equation*}
where indices with $x_{i+1}=x_i$ contribute nothing.  
As we assume that $S \ne \emptyset$ and $s_t \ge 0$ for all $t$ such that $x_{t+1} \ne x_t$, it follows that $s^*=\sup S \ge l_0 = 0$. Since each $s_i \in S$, it follows that $l_t\le s^*$ for all $t$.  
Thus $\{l_t\}_{t \ge 0}$ is a nondecreasing sequence bounded above by $s^*$, hence convergent with some finite limit
\begin{equation*}
\bar l := \lim_{t\to\infty} l_t \le s^*.
\end{equation*}
To see $\bar l = s^*$, let $\epsilon>0$. By definition of supremum there exists an index $j$ with $s_j> s^*-\epsilon$. For all $t\ge j$ we then have $l_t\ge s_j> s^*-\epsilon$, hence $\bar l \ge s^*-\epsilon$. Letting $\epsilon\to0$ yields $\bar l=s^*$. Since we assume that $S \ne \emptyset$, \cref{empirical_lipschitz_defin} yields
\begin{equation*}
    L^e_F(\left\{x_t\right\}_{t \ge 0}) = \sup S = s^*,
\end{equation*}
and hence
\begin{equation*}
    \lim_{t\to\infty} l_t = L^e_F(\left\{x_t\right\}_{t \ge 0}).
\end{equation*}

\medskip\noindent\textbf{Case 2: $s^* = +\infty$.}  
For any $C>0$ the definition of supremum (being $+\infty$) yields some index $j$ with $s_j>C$. Then for all $t\ge j$ we have $l_t\ge s_j>C$. As this holds for every $C>0$, 
\begin{equation*}
\lim_{t\to\infty} l_t=+\infty = s^*,
\end{equation*}
which again equals $L^e_F(\{x_t\}_{t\ge0})$.

\medskip

Combining the above cases completes the proof: $\{l_t\}$ is nondecreasing and 
\begin{equation*}
\lim_{t\to\infty} l_t = \sup S = L^e_F(\{x_t\}_{t\ge0}).
\end{equation*}
\end{proof}

\section{Sampling Mechanism Proofs}
\begin{proof}[Proof of {\cref{lem:1}}]
Note that by definition \eqref{eq:error matrix}, $\mathbb{E}[\delta_t] = 0$. Additionally, note  that $\left[\delta_t\right]_{i,j} = \dfrac{1}{\theta_t}\sum\limits_{k=1}^{\theta_t} M^k_{i,j} - [\mathbb{E}[M]]_{i,j}$, and that the matrices $M^k$ are i.i.d., hence $\mathrm{Var}\left[\left[\delta_t\right]_{i,j}\right] = \dfrac{1}{\theta_t} \mathrm{Var} \left[M_{i,j} - [\mathbb{E}[M]]_{i,j}\right]$. Variance is invariant under translations on the real line, so
\begin{equation*}
     \mathrm{Var}[[\delta_t]_{i,j}] = \dfrac{1}{\theta_t} \mathrm{Var}[M_{i,j}].
\end{equation*}
Subsequently, using Chebyshev's inequality we have that
\begin{equation*}
    \mathbb{P}(|[\delta_t]_{i,j}| > \eta) \le \dfrac{\mathrm{Var}[[\delta_t]_{i,j}]}{\eta^2} = \dfrac{\mathrm{Var}[M_{i,j}]}{\theta_t \eta^2}.
\end{equation*}
Thus, choosing $\theta_t = t ^ {2 + \epsilon}$, $\eta = \dfrac{1}{t^{0.5 + 0.25 \epsilon}}$, and applying Chebyshev's inequality yields
\begin{equation}\label{eq:1085}
     \mathbb{P}\left(|[\delta_t]_{i,j}| > \dfrac{1}{t^{0.5 + 0.25 \epsilon}} \right) \le \dfrac{\mathrm{Var}[M_{i,j}]}{t^{1+0.5\epsilon}}.
\end{equation}

Consider the \textit{element-wise} matrix norm $\|\delta_t\|_{1,1} = \sum_{i=1}^n \sum_{j=1}^n |[\delta_t]_{i,j}|$. If $\|\delta_t\|_{1,1} > \dfrac{n^2}{t^{0.5 + 0.25\epsilon}}$, then by the pigeonhole principle, at least one of the events $|[\delta_t]_{i,j}| > \dfrac{1}{t^{0.5 + 0.25 \epsilon}}$ occurred. Therefore, 
\begin{equation*}
    \mathbb{P}\left(\|\delta_t\|_{1,1} > \dfrac{n^2}{t^{0.5 + 0.25\epsilon}} \right) \le \mathbb{P} \left( \bigcup_{i,j} \left\{|[\delta_t]_{i,j}| > \dfrac{1}{t^{0.5 + 0.25 \epsilon}}\right\} \right).
\end{equation*}
Consequently, using the union bound together with \eqref{eq:1085} implies that
\begin{equation*}
    \mathbb{P}\left(\|\delta_t\|_{1,1} > \dfrac{n^2}{t^{0.5 + 0.25\epsilon}} \right) \le \dfrac{1}{t^{1+0.5\epsilon}}\sum_{i=1}^n \sum_{j=1}^n \mathrm{Var}[M_{i,j}].
\end{equation*}

By the equivalence between norms in $\real^{n\times n}$, there exist $a,A > 0$, such that $ a \cdot \|\delta_t\|_{1,1} \le \|\delta_t\|_{2,2} \le A \cdot \|\delta_t\|_{1,1}$,
where $\|\delta_t\|_{2,2}$ is the induced $\ell_2$ norm on the matrix $\delta_t$, given by $\|\delta_t\|_{2,2} = \max\limits_{\|x\|_2=1} \|\delta_t x\|_2$. 
It follows that
\begin{equation*}
    \mathbb{P}\left(\|\delta_t\| > \dfrac{A\cdot n^2}{t^{0.5 + 0.25\epsilon}} \right) \le \dfrac{1}{t^{1+0.5\epsilon}}\sum_{i=1}^n \sum_{j=1}^n \mathrm{Var}[M_{i,j}].
\end{equation*}
Finally, setting $r = \sum_{i=1}^n \sum_{j=1}^n \mathrm{Var}[M_{i,j}]$ and using the fact that 
\begin{equation*}
    \sum_{i=1}^n \sum_{j=1}^n \mathrm{Var}[M_{i,j}]
\end{equation*}
is finite (cf. \cref{assum:blanket}), we conclude that 
\begin{equation*}
    \sum_{t=1}^\infty  \mathbb{P} \left(\|\delta_t\| > \dfrac{A\cdot n^2}{t^{0.5 + 0.25 \epsilon}} \right) \le  \sum_{t=1}^\infty  \dfrac{r}{t^{1+0.5\epsilon}} < \infty,
\end{equation*}
and subsequently, the lemma follows from the Borel--Cantelli Theorem. 
\end{proof}

\begin{proof}[Proof of {\cref{lem:2}}]
Since $M^t_{i,j}$ is sub-Gaussian, $M^t_{i,j} - \mathbb{E}[M]_{i,j}$ and $\dfrac{1}{\theta_t}(M^t_{i,j} - \mathbb{E}[M]_{i,j})$ are also sub-Gaussian.
Recall that $\delta_t = \bar{M}^t - \mathbb{E}[M] = \sum_{l=1}^{\theta_t} \dfrac{1}{\theta_t}(M^l - \mathbb{E}[M])$,
where $\bar{M}^t = \dfrac{1}{\theta_t} \sum\limits_{l=1}^{\theta_t} M^l$ is the estimator of $\mathbb{E}[M]$.
Hence,
\begin{align*}
    \mathbb{P} \left( \left|[\delta_t]_{i,j} \right| > \dfrac{1}{t^{0.5 + 0.25 \cdot \epsilon}} \right) &= \mathbb{P} \left( \left|\bar{M}^t_{i,j} - \mathbb{E}[M]_{i,j} \right| > \dfrac{1}{t^{0.5 + 0.25 \cdot \epsilon}} \right) \\&= \mathbb{P} \left( \left|\sum_{l=1}^{\theta_t} \dfrac{1}{\theta_t}(M^l_{i,j} - \mathbb{E}[M]_{i,j}) \right| > \dfrac{1}{t^{0.5 + 0.25 \cdot \epsilon}} \right).
\end{align*}

If $\mathbb{P}\left([\delta_t]_{i,j} \ne 0\right) = 0$, then trivially for any $C>0$ it holds that \begin{equation*}
    \mathbb{P} \left( \left|[\delta_t]_{i,j} \right| > \dfrac{1}{t^{0.5 + 0.25 \cdot \epsilon}} \right) \le 2 \exp \left(-C \cdot t^{0.5 \epsilon} \right).
\end{equation*}
Otherwise, if $\mathbb{P}\left([\delta_t]_{i,j} \ne 0\right) > 0$, then $\left\Vert\dfrac{1}{\theta_t} (M^l_{i,j} - \mathbb{E}[M]_{i,j})\right\Vert_{\psi_2} > 0$, and we have by the General Hoeffding's Inequality with $m=\theta_t$ and $k=1/t^{0.5 + 0.25\epsilon}$ that
\begin{align*}
    \mathbb{P} \left( \left|[\delta_t]_{i,j} \right| > \dfrac{1}{t^{0.5 + 0.25 \cdot \epsilon}} \right) &= \mathbb{P} \left( \left|\sum_{l=1}^{\theta_t} \dfrac{1}{\theta_t}(M^l_{i,j} - \mathbb{E}[M]_{i,j}) \right| > \dfrac{1}{t^{0.5 + 0.25 \cdot \epsilon}} \right) \\&\le 2 \exp \left(-\dfrac{c \cdot \dfrac{1}{t^{1+0.5\epsilon}}}{\sum_{l=1}^{\theta_t} \left\Vert\dfrac{1}{\theta_t} (M^l_{i,j} - \mathbb{E}[M]_{i,j})\right\Vert_{\psi_2}^2} \right) \\&= 2 \exp \left(-\dfrac{c \cdot \dfrac{1}{t^{1+0.5\epsilon}}}{\theta_t \cdot \dfrac{1}{\theta_t^2} \left\Vert M^0_{i,j} - \mathbb{E}[M]_{i,j}\right\Vert_{\psi_2}^2} \right),
\end{align*}
where the last inequality follows from the assumption that $\{M^l\}_{l\geq 0}$ are i.i.d. and the norm property $\|\lambda \xi\| = |\lambda| \|\xi\|$.

Denoting $C = \dfrac{c}{\|M^0_{i,j} - \mathbb{E}[M]_{i,j}\|_{\psi_2}^2} $
and rearranging the elements we obtain that
\begin{equation*}
    \mathbb{P} \left( \left|[\delta_t]_{i,j} \right| > \dfrac{1}{t^{0.5 + 0.25 \cdot \epsilon}} \right) \le 2 \exp \left(-C \cdot \dfrac{\theta_t}{t^{1+0.5\epsilon}} \right).
\end{equation*}
By the choice of $\theta_t$ we then have that
\begin{equation*}
    \mathbb{P} \left( \left|[\delta_t]_{i,j} \right| > \dfrac{1}{t^{0.5 + 0.25 \cdot \epsilon}} \right) \le 2 \exp \left(-C \cdot t^{0.5\epsilon} \right).
\end{equation*}
Thus, whether $\mathbb{P}\left([\delta_t]_{i,j} \ne 0\right) = 0$ or $\mathbb{P}\left([\delta_t]_{i,j} \ne 0\right) > 0$, there exists a constant $C>0$, such that 
\begin{equation*}
    \mathbb{P} \left( \left|\bar{M}^t_{i,j} - \mathbb{E}[M]_{i,j} \right| > \dfrac{1}{t^{0.5 + 0.25 \cdot \epsilon}} \right) \le 2 \exp \left(-C \cdot t^{0.5 \epsilon} \right).
\end{equation*}

Consider the entry-wise $\ell_1$ norm of $\delta_t = \bar{M}^t - \mathbb{E}[M]$ given by $\|\delta_t\|_{1,1} = \sum_{i=1}^n \sum_{j=1}^n \left| \bar{M}^t_{i,j} - \mathbb{E}[M]_{i,j} \right|. $
Note that for the event $\|\delta_t\|_{1,1} > \dfrac{n^2}{t^{0.5 + 0.25 \cdot \epsilon}} $
to occur, by the pigeonhole principle there has to be at least one pair of indices ${i,j}$ such that $\left|\bar{M}^t_{i,j} - \mathbb{E}[M]_{i,j} \right| > \dfrac{1}{t^{0.5 + 0.25 \cdot \epsilon}}$. 
Therefore, 
\begin{equation*}
    \mathbb{P} \left( \|\delta_t\|_{1,1} > \dfrac{n^2}{t^{0.5 + 0.25 \cdot \epsilon}} \right) \le \mathbb{P} \left( \bigcup_{i,j} \left\{\left|\bar{M}^t_{i,j} - \mathbb{E}[M]_{i,j} \right| > \dfrac{1}{t^{0.5 + 0.25 \cdot \epsilon}}\right\} \right).
\end{equation*}
Using the union bound, it follows that 
\begin{equation*}
    \mathbb{P} \left( \|\delta_t\|_{1,1} > \dfrac{n^2}{t^{0.5 + 0.25 \cdot \epsilon}} \right) \le 2n^2 \exp \left(-C t^{0.5 \epsilon} \right).
\end{equation*}
As stated in \cref{lem:1}, by the equivalence between norms in $\real^{n\times n}$, there exist $a,A > 0$, such that $ a \cdot \|\delta_t\|_{1,1} \le \|\delta_t\|_{2,2} \le A \cdot \|\delta_t\|_{1,1}. $
Thus, it follows that
\begin{equation*}
    \mathbb{P} \left( \|\delta_t\| > \dfrac{A\cdot n^2}{t^{0.5 + 0.25 \cdot \epsilon}} \right) \le 2n^2 \exp \left(-C t^{0.5 \epsilon} \right).
\end{equation*}
Summing over $t$ then yields that
\begin{equation*}
    \sum_{t=1}^\infty \mathbb{P} \left(\|\delta_t\| > \dfrac{A\cdot n^2}{t^{0.5 + 0.25 \cdot \epsilon}}  \right) \le \sum_{t=1}^\infty 2 n^2 \cdot \exp \left(-C t^{0.5 \epsilon} \right) < \infty,
\end{equation*}
and the result follows from the Borel--Cantelli Theorem. 
\end{proof}

\begin{proof}[Proof of {\cref{lem:3}}] 
The proof follows immediately from the fact that the function that returns the minimum eigenvalue of a symmetric matrix, $\lambda_{min}: S^n \rightarrow \real$, is continuous, and the fact that $\sigma > 0$. The continuity of $\lambda_{min}$ follows from \cite[Theorem 2.4.9.2]{Horn_Johnson_1985}.
\end{proof}

\section{Technical Results}
\label{app:sec:technical}
\begin{proof}[Proof of {\cref{cumulative_matrix_error_effects_are_finite}}]
	Recall that
	\begin{equation*}
		\mathcal{L}_{\beta}(x, y, z; A) = h(x) + P(y) - \langle z, Ax - y \rangle + \dfrac{\beta}{2}\|Ax - y\|^2,
	\end{equation*}
	which yields that
	\begin{equation*}
		e_{t+1} = \langle z_t, \bar{M}^t x_t - \bar{M}^{t+1}x_t \rangle + \dfrac{\bar{\beta}}{2}\left(\|\bar{M}^{t+1} x_t - y_t\|^2 - \|\bar{M}^{t} x_t - y_t\|^2\right).
	\end{equation*}

	Since 
	\begin{align*}
		\|\bar{M}^{t+1} x_t &- y_t\|^2 - \|\bar{M}^{t} x_t - y_t\|^2 \\ =& \|\bar{M}^{t+1} x_t\|^2 - 2x_t^T \left(\bar{M}^{t+1}\right)^T y_t + \|y_t\|^2 - \|\bar{M}^{t} x_t\|^2 \\&+ 2x_t^T \left(\bar{M}^{t}\right)^T y_t - \|y_t\|^2 \\ =& \|\bar{M}^{t+1} x_t\|^2 - \|\bar{M}^{t} x_t\|^2 - 2x_t^T \left(\bar{M}^{t+1} - \bar{M}^t\right)y_t.
	\end{align*}
	
	It follows that
	\begin{align}\label{eq:1010}
		e_{t+1} &= \langle z_t,\bar{M}^t x_t - \bar{M}^{t+1}x_t \rangle \\&+ \dfrac{\bar{\beta}}{2}\left(\|\bar{M}^{t+1} x_t\|^2 - \|\bar{M}^{t} x_t\|^2 - 2x_t^T \left(\bar{M}^{t+1} - \bar{M}^t\right)y_t\right).\nonumber
	\end{align}
	Utilizing the fact that $\bar{M}^{t+1} = \mathbb{E}[M] + \delta_{t+1}$, $\bar{M}^{t} = \mathbb{E}[M] + \delta_{t}$, we have that
	\begin{align*}
		\|\bar{M}^{t+1} x_t\|^2 &- \|\bar{M}^{t} x_t\|^2 \\ =& \|\mathbb{E}[M] x_t + \delta_{t+1}x_t\|^2 - \|\mathbb{E}[M] x_t + \delta_{t}x_t\|^2 \\ =&
		\|\mathbb{E}[M] x_t\|^2 + 2x_t^T \mathbb{E}[M]^T \delta_{t+1} x_t + \|\delta_{t+1}x_t\|^2 - \|\mathbb{E}[M] x_t\|^2 \\&- 2 x_t^T \mathbb{E}[M]^T \delta_{t} x_t - \|\delta_t x_t\|^2 \\ =& 2x_t^T \mathbb{E}[M]^T \left(\delta_{t+1} -\delta_t\right) x_t + \|\delta_{t+1}x_t\|^2 - \|\delta_t x_t\|^2.
	\end{align*}
	Plugging it into \eqref{eq:1010} then results in
	\begin{align*}
		e_{t+1} = \langle z_t,\bar{M}^t x_t - \bar{M}^{t+1}x_t \rangle &+ \dfrac{\bar{\beta}}{2}(2x_t^T \mathbb{E}[M]^T \left(\delta_{t+1} -\delta_t \right) x_t \\
		&+ \|\delta_{t+1}x_t\|^2 - \|\delta_t x_t\|^2 - 2x_t^T \left(\bar{M}^{t+1} - \bar{M}^t\right)y_t).
	\end{align*}
	
	Using the fact that 
	\begin{equation*}
		\delta_{t+1} - \delta_t = \bar{M}^{t+1} - \mathbb{E}[M] - \left(\bar{M}^t - \mathbb{E}[M]\right) = \bar{M}^{t+1} - \bar{M}^t,
	\end{equation*}
	we can update our previous equation for $e_{t+1}$ to
	\begin{align}\label{eq:50}
		e_{t+1} = \langle z_t,\bar{M}^t x_t - \bar{M}^{t+1}x_t \rangle &+ \dfrac{\bar{\beta}}{2}(2x_t^T \mathbb{E}[M]^T \left(\bar{M}^{t+1} - \bar{M}^t \right) x_t \\ &+ \|\delta_{t+1}x_t\|^2 - \|\delta_t x_t\|^2 - 2x_t^T \left(\bar{M}^{t+1} - \bar{M}^t \right)y_t ) \nonumber.
	\end{align}
	
	We denote by $\hat{M}^{t+1}$ the average of the $\theta_{t+1}-\theta_t$ samples taken in round $t+1$. Formally,
	\begin{equation*}
		\hat{M}^{t+1} = \dfrac{1}{\theta_{t+1} - \theta_t} \sum\limits_{i=\theta_t + 1}^{\theta_{t+1}} M^i.
	\end{equation*}
	Furthermore, we denote $\hat{\delta_t} = \hat{M}^t - \mathbb{E}[M]$. By the update rule of $\bar{M}^{t+1}$ and the definition of $\hat{\delta_t}$ and $\delta_t$
	\begin{equation}\label{eq:1080}
		\begin{aligned}
			\bar{M}^{t+1} - \bar{M}^t  = \dfrac{\theta_t}{\theta_{t+1}}\bar{M}^t + \dfrac{\theta_{t+1}-\theta_t}{\theta_{t+1}}\hat{M}^{t+1} - \bar{M}^t  =& \dfrac{\theta_{t+1}-\theta_t}{\theta_{t+1}} \left(\hat{M}^{t+1} - \bar{M}^{t} \right) \\ 
			=& \dfrac{\theta_{t+1}-\theta_t}{\theta_{t+1}}\left(\hat{\delta}_{t+1} - \delta_{t}\right).
		\end{aligned}
	\end{equation}
	To reduce clutter, we denote 
	\begin{equation*}
		s_{t+1} = \dfrac{\theta_{t+1}-\theta_t}{\theta_{t+1}}.
	\end{equation*}
	Introducing this notation to \eqref{eq:1080}, we derive
	\begin{equation}\label{eq:51}
		\bar{M}^{t+1} - \bar{M}^t = s_{t+1} \left( \hat{\delta}_{t+1} - \delta_t \right).
	\end{equation}
	Combining \eqref{eq:50} and \eqref{eq:51} results in
	\begin{align*}
		e_{t+1} &= \langle z_t,-s_{t+1}(\hat{\delta}_{t+1} - \delta_{t})x_t \rangle \\&+ \dfrac{\bar{\beta}}{2} \left(2x_t^T \mathbb{E}[M]^T \left(s_{t+1} \left(\hat{\delta}_{t+1} - \delta_{t} \right)\right) x_t \right.\\&\left.+ \|\delta_{t+1}x_t\|^2 - \|\delta_t x_t\|^2 - 2x_t^T \left(s_{t+1}\left(\hat{\delta}_{t+1} - \delta_{t}\right)\right)y_t\right),
	\end{align*}
	which in particular implies that
	\begin{align*}
		e_{t+1} &\le \langle z_t,-s_{t+1} \left(\hat{\delta}_{t+1} - \delta_{t} \right)x_t \rangle \\&+ \dfrac{\bar{\beta}}{2}\left(2x_t^T \mathbb{E}[M]^T \left(s_{t+1}\left(\hat{\delta}_{t+1} - \delta_{t}\right)\right) x_t \right.\\&+\left. \|\delta_{t+1}x_t\|^2 - 2x_t^T \left(s_{t+1}\left(\hat{\delta}_{t+1} - \delta_{t}\right)\right)y_t\right).
	\end{align*}
	By the Cauchy--Schwarz and the triangle inequalities,
	\smaller
	\begin{align*}
		e_{t+1} & \le s_{t+1}\| z_t\|\left(\|\hat{\delta}_{t+1}\| + \|\delta_{t}\|\right)\|x_t\| \\&+ s_{t+1} \dfrac{\bar{\beta}}{2}\left(\|\mathbb{E}[M]x_t\| \left(\|\hat{\delta}_{t+1}\| + \|\delta_{t}\| \right) \|x_t\| \right.\\&\left.+ 2\|x_t\| \left(\|\hat{\delta}_{t+1}\| + \|\delta_{t}\| \right) \|y_t\| + \|\delta_{t+1}\|^2\cdot \|x_t\|^2\right).
	\end{align*}
	\normalsize
	
	By \Cref{assum:2}, there exists $D>0$, such that 
    \begin{equation*}
        D \ge \sup\limits_{t \ge 0} \{\max \{ \|x_t\|, \|\mathbb{E}[M] x_t\|, \|y_t\|, \|z_t\|\}\}.
    \end{equation*} 
	Using this fact, we have that
	\smaller
	\begin{equation}\label{eq:1086}
		\begin{aligned}
			e_{t+1} & \le s_{t+1}D^2 \left(\|\hat{\delta}_{t+1}\| + \|\delta_{t}\| \right) \\&+ s_{t+1}D^2 \dfrac{\bar{\beta}}{2} \left( \left(\|\hat{\delta}_{t+1}\| + \|\delta_{t}\|\right) + 2 \left(\|\hat{\delta}_{t+1}\| + \|\delta_{t}\| \right) + \|\delta_{t+1}\|^2 \right) \\&= \left(\dfrac{3\bar{\beta}}{2} + 1 \right)D^2 s_{t+1} \left(\|\hat{\delta}_{t+1}\| + \|\delta_t\| \right) + \dfrac{\bar{\beta}}{2} D^2 s_{t+1} \|\delta_{t+1}\|^2.
		\end{aligned}
	\end{equation}
	\normalsize
	Let $\kappa \in \{\epsilon, 1 + \epsilon\}$ be chosen according to the sampling regime described by \Cref{definition:2}. 
	By the Taylor series expansion of $(t+1)^{1+\kappa}$,
	\begin{equation*}
		(t+1)^{1+\kappa} = t^{1+\kappa} + \dfrac{1}{1!}(t+1-t)(1+\kappa)t^{\kappa} + O(t^{-1 + \kappa}) = t^{1+\kappa} + (1+\kappa)t^{\kappa} + O(t^{-1 + \kappa}). 
	\end{equation*}
	Therefore, there exist constants $A_1, A_2 > 0$, such that for all $t \ge 1$, it holds that $  t^{1+\kappa} + A_1 t^{\kappa} \le (t+1)^{1+\kappa} \le t^{1+\kappa} + A_2 t^{\kappa}.$
	Hence, it follows that
	\begin{align}\label{eq:1090}
		s_{t+1} &= \dfrac{\theta_{t+1}-\theta_t}{\theta_{t+1}} = \dfrac{(t+1)^{1+\kappa}-t^{1+\kappa}}{(t+1)^{1+\kappa}} \nonumber \\&\le \dfrac{t^{1+\kappa} + A_2 t^{\kappa}-t^{1+\kappa}}{(t+1)^{1+\kappa}} = \dfrac{A_2 t^\kappa}{(t+1)^{1+\kappa}} \\&\le \dfrac{A_2 t^\kappa}{t^{1+\kappa}} = \dfrac{A_2}{t}. \nonumber
	\end{align}
	Now our mission boils down to producing (probabilistic) bounds on $\|\delta_t\|$ and $\|\hat{\delta}_{t}\|$.
	By \Cref{lem:1} and \Cref{lem:2}, $\mathbb{P} \left(\|\delta_t\| > O \left(\dfrac{1}{t^{0.5 + 0.25\epsilon}} \right) \; i.o. \right) = 0$. Therefore, with probability 1, there exists an index $\tau_1$ such that for all $t>\tau_1$, $\|\delta_t\| \le \dfrac{1}{\sqrt{t}}$.

	Using similar arguments to those used in \Cref{lem:1} and \Cref{lem:2}, it can also be shown that $\mathbb{P} \left(\|\hat{\delta}_{t}\| > \dfrac{\alpha_1}{t^{0.25\epsilon}} \; i.o.\right) = 0$ for an $\alpha_1 > 0$.
	We provide a sketch of the formal adaptations required below.
	\begin{itemize}
		\item If \Cref{assum:4} holds true, the proof is almost identical to \Cref{lem:2}, with the following differences:
		\begin{enumerate}
			\item We use $\theta_{t+1}-\theta_t$ instead of $\theta_t$ in the definition of $\hat{\delta}_t$:
			\begin{equation*}
				\hat{\delta}_{t+1} = \hat{M}^{t+1} - \mathbb{E}[M] = \sum\limits_{l = \theta_t + 1}^{\theta_{t+1}}\left(\dfrac{1}{\theta_{t+1}-\theta_t} \left(M^l - \mathbb{E}[M]\right)\right).
			\end{equation*}
			\item Subsequently, we derive an upper bound for 
			\begin{equation*}
				\mathbb{P} \left( |[\hat{\delta}_{t+1}]_{i,j}| > \dfrac{1}{t^{0.25 \epsilon}}\right),
			\end{equation*}
			instead of 
			\begin{equation*}
				\mathbb{P} \left( \left|[\delta_t]_{i,j} \right| > \dfrac{1}{t^{0.5 + 0.25 \cdot \epsilon}} \right).
			\end{equation*}
			\item Before deriving 
			\begin{equation*}
				\mathbb{P} \left( \left|[\delta_t]_{i,j} \right| > \dfrac{1}{t^{0.5 + 0.25 \cdot \epsilon}} \right) \le 2 \exp \left(-C \cdot t^{0.5\epsilon} \right),
			\end{equation*} 
			we use the lower bound for $\theta_{t+1}-\theta_t$ under \Cref{assum:4}, $A_1 t^\epsilon$.
		\end{enumerate}
		
		\item If \Cref{assum:4} does not hold, the proof is almost identical to \Cref{lem:1}, with the following differences:
		\begin{enumerate}
			\item We replace $\delta_t$ with
			\begin{equation*}
				\hat{\delta}_{t+1} := \hat{M}^{t+1} - \mathbb{E}[M] = \dfrac{1}{\theta_{t+1}-\theta_t} \sum\limits_{l = \theta_t + 1}^{\theta_{t+1}}\left(M^l - \mathbb{E}[M]\right).
			\end{equation*}
			
			\item Subsequently, the expression for $\mathrm{Var}[[\delta_t]_{i,j}]$ in 
			\begin{equation*}
				\mathrm{Var}[[\delta_t]_{i,j}] = \dfrac{1}{\theta_t} \mathrm{Var}[M_{i,j}]
			\end{equation*} 
			will be replaced by
			\begin{equation*}
				\mathrm{Var}[[\hat{\delta}_{t+1}]_{i,j}] = \dfrac{1}{\theta_{t+1} - \theta_t} \mathrm{Var}[M_{i,j}].
			\end{equation*}
			
			\item In the application of Chebyshev's inequality that follows immediately afterwards, we modify our choice of $\eta$, from $\dfrac{1}{t^{0.5 + 0.25\epsilon}}$ to $\dfrac{1}{t^{0.25\epsilon}}$.
			
			\item Finally, we use the lower bound for $\theta_{t+1}-\theta_t$ if \Cref{assum:4} does not hold, $A_1 t^{1+\epsilon}$, yielding the replacement 
			\begin{equation*}
				\mathbb{P}\left(|[\hat{\delta}_{t+1}]_{i,j}| > \dfrac{1}{t^{0.25 \epsilon}} \right) \le \dfrac{\mathrm{Var}[M_{i,j}]}{A_1 t^{1+0.5\epsilon}}
			\end{equation*}
			instead of the original expression 
			\begin{equation*}
				\mathbb{P}\left(|[\delta_t]_{i,j}| > \dfrac{1}{t^{0.5 + 0.25 \epsilon}} \right) \le \dfrac{\mathrm{Var}[M_{i,j}]}{t^{1+0.5\epsilon}}.
			\end{equation*}
		\end{enumerate}
	\end{itemize}
	Since the proof process is identical to the one shown in \Cref{lem:1} and \Cref{lem:2}, other than the minor differences pointed out above, we do not provide the complete proof.
	
	Returning to the main stream of the proof, we have, due to 
    \begin{equation*}
        \mathbb{P} \left(\|\hat{\delta}_{t}\| > \dfrac{\alpha_1}{t^{0.25\epsilon}}\; i.o.\right) = 0
    \end{equation*}
    justified above, that there exists with probability 1 an index $\tau_2$, so that for all $t > \tau_2$
	\begin{equation}\label{eq:1095}
		\|\hat{\delta}_{t}\| \le \dfrac{\alpha_1}{t^{0.25 \epsilon}}.
	\end{equation}
	
	Utilizing \eqref{eq:1086}, \eqref{eq:1090} and \eqref{eq:1095}, we can show that for all sufficiently large $t$, $e_{t+1} \le O\left(\dfrac{1}{t^{1 + 0.25 \epsilon}}\right)$. 
	Indeed, for any $t > \tau \equiv \max\{\tau_1, \tau_2\}$, there exist constants $\alpha_1, \alpha_2$ so that
	\begin{equation*}
		e_{t+1} \le \left(\dfrac{3 \bar{\beta}}{2} + 1 \right)D^2 \dfrac{A_2}{t} \left(\dfrac{\alpha_1}{t^{0.25 \epsilon}} + \dfrac{\alpha_2}{\sqrt{t}} \right) + \dfrac{\bar{\beta}D^2}{2} \dfrac{A_2}{t} \cdot \dfrac{\alpha_2^2}{t},
	\end{equation*}
	where $\alpha_1$, $\alpha_2$ are the scalars fulfilling $\|\hat{\delta}_t\| \le \dfrac{\alpha_1}{t^{0.25\epsilon}}$ and $\|\delta_t\| \le \dfrac{\alpha_2}{t^{0.5 + 0.25\epsilon}}$.
	
	Consequently, it follows  with probability $1$ that $\sum\limits_{t=1}^\infty e_t < \infty. $
\end{proof}





\textbf{Acknowledgments.}
This work was supported by the ISRAEL SCIENCE FOUNDATION (grant No. 637/21).


\bibliography{AL_bib} 


\end{document}